\documentclass[12pt,a4paper]{amsart}

\usepackage[latin1]{inputenc}
\usepackage{amsthm, textcomp, amsmath, amsfonts,  amssymb }
\usepackage[pdftex]{graphicx}
\usepackage{url}
\usepackage{subfig}
\usepackage[colorlinks]{hyperref}
\usepackage[a4paper,left=3cm,right=3cm]{geometry}
\allowdisplaybreaks[4]

\DeclareMathOperator{\M}{\mathcal{M}}
\DeclareMathOperator{\Y}{\mathbb{Y}}
\DeclareMathOperator{\V}{\mathcal{V}}
\DeclareMathOperator{\R}{\mathbb{R}}
\DeclareMathOperator{\diam}{diam}

\DeclareMathOperator{\X}{\mathbb{X}}
\begin{document}

\title[Optimal cubature formulas on manifolds]{Asymptotically optimal cubature formulas on manifolds for prefixed weights}

\author[M. Ehler]{Martin Ehler}
\address[M. Ehler]{Department of Mathematics, University of Vienna, Oskar-Morgenstern-Platz 1, A-1090 Vienna, Austria}
\email{martin.ehler@univie.ac.at}

\author[U. Etayo]{Uju\'e Etayo}
\address[U. Etayo]{5010 Institute of Analysis and Number Theory 8010 Graz, Kopernikusgasse 24/II}
\email{etayo@math.tugraz.at}

\author[B. Gariboldi]{Bianca Gariboldi}
\address[B. Gariboldi]{Dipartimento di Ingegneria Gestionale, dell'Informazione e della Produzione,
Universit\`a degli Studi di Bergamo, Viale Marconi 5, Dalmine BG, Italy}
\email{biancamaria.gariboldi@unibg.it}

\author[G. Gigante]{Giacomo Gigante}
\address[G. Gigante]{Dipartimento di Ingegneria Gestionale, dell'Informazione e della Produzione,
Universit\`a degli Studi di Bergamo, Viale Marconi 5, Dalmine BG, Italy}
\email{giacomo.gigante@unibg.it}

\author[T. Peter]{Thomas Peter}
\address[T. Peter]{Department of Mathematics, University of Vienna, Oskar-Morgenstern-Platz 1, A-1090 Vienna, Austria}
\email{thomas.peter@univie.ac.at}

\begin{abstract}
We propose a new extension of the proof of the Korevaar-Meyers conjecture by Bondarenko, Radchenko and Viazovska for cubature formulas with prefixed weights (we fix different weights for different points) on algebraic and Riemannian manifolds.
\end{abstract}



\thanks{Martin Ehler has been supported by the Vienna Science and
Technology Fund WWTF project VRG12-009.}
\thanks{Uju\'e Etayo has been supported by the Austrian Science Fund FWF project F5503 (part of the Special Research Program (SFB) Quasi-Monte Carlo Methods: Theory and Applications), by MTM2017-83816-P from Spanish Ministry of Science MICINN and by 21.SI01.64658 from Universidad de Cantabria and Banco de Santander.
}
\thanks{Bianca Gariboldi and Giacomo Gigante have been supported by an Italian GNAMPA 2019 project.}

\maketitle

\newtheorem{Theorem}{Theorem}
\newtheorem{Cor}[Theorem]{Corollary}
\newtheorem{lemma}[Theorem]{Lemma}
\newtheorem{definition}[Theorem]{Definition}
\newtheorem{prop}[Theorem]{Proposition}
\newtheorem{ex}[Theorem]{Example}
\newtheorem{remark}[Theorem]{Remark}

\setcounter{tocdepth}{1}
\tableofcontents

\section{Introduction}

Briefly speaking, a set of points $\{x_{j}\}_{j=1}^{N}$ with positive weights $\{\omega_{j} \}_{j=1}^{N}$ is a cubature of strength $L$ if all polynomials of degree $L$ are exactly integrated by the weighted sums over sampling values (a more detailed definition will be given in Section \ref{sec:def}).
Cubature formulas have been deeply studied from both theoretical and practical points of view.
Their associated literature is extense, we refer for example to the books \cite{sobolev2013theory, ST} and references therein.
Not many explicit examples of cubature formulas are known, for a recopilation of cubature formulas, one can check \cite{COOLS2003445}.
The most studied case is the one of the so called $L$-designs, where all the weights are equal, see \cite{BRV,DGS,KM}.

Here we care about the existence of cubature formulas from a non-constructive point of view.
We know that the existence of a cubature of strength $L$ depends on a relation between the dimension $d$ of the manifold, the degree $L$ of the polynomials and the number $N$ of points and weights. 
In particular, if $C_{\mathcal{M}}$ is a constant depending only on the manifold $\mathcal{M}$ and $N\geq C_{\mathcal{M}}L^{d}$ then there exists a cubature of strength $L$ for many smooth compact manifolds $\mathcal{M}$, see \cite{HP}. 

If weights are fixed a priori, then the existence of cubature points is more complicated. 
Bondarenko, Radchenko and Viazovska showed in  \cite{BRV} that there is a constant $C_{d}$ such that for every $N \geq C_{d}L^{d}$ there exists an $L$-design in the $d$-dimensional sphere with exactly $N$ nodes. 
Later, Etayo, Marzo and Ortega-Cerd\`{a}  generalized this result in \cite{EMOC} to the case of a compact connected affine algebraic manifold. 
Gariboldi and Gigante proved the analogous result on a compact connected oriented Riemannian manifold, when polynomials are replaced by eigenfunctions of the Laplace-Beltrami operator, see \cite{GG}. 
For the sphere, each eigenfunction of the Laplace-Beltrami operator is the restriction of a polynomial, so that both, \cite{EMOC} and \cite{GG} apply. 
In general though, the eigenfunctions of the Laplace-Beltrami operator on an algebraic manifold are not necessarily polynomials, see Section \ref{elipse} for the ellipse. 

In this paper we study the existence of cubature points for fixed weights that are not all equal for the previous cases: compact connected oriented Riemannian manifolds with Laplacian eigenfunctions and compact connected real algebraic manifolds with algebraic polynomials. 
In order to do so, we prove the existence of weighted area partitions on manifolds and some Marcinkiewicz-Zygmund type inequalities for gradients of polynomials and diffusion polynomials.

\subsection{Organization}
We define cubature formulas on Riemannian and algebraic manifolds in Section \ref{sec:def}. 
In Section \ref{SecMR} we state the Main result and in Section \ref{SecIn} we provide the proof based on an existence result from Brouwer degree theory. The remaining part of the manuscript is dedicated to verify that the assumptions of the Brouwer degree theorem are satisfied. 
In particular, in Section \ref{SecP1} we prove a result that may be of independent interest: the existence of a partition of any Riemannian manifold into parts with given prefixed areas.
In Appendix  \ref{SecM} we present some of the properties of Riemannian manifolds that are used in the proofs.


\section{Cubature formulas on Riemannian and on algebraic manifolds}\label{sec:def}

\subsection{Cubature formulas on Riemannian manifolds}\label{sec:M}

Let $\mathcal{M}$ be a $d$-dimensional connected compact orientable Riemannian manifold without boundary, where the measure $\mu_{\M}$ induced by the Riemannian metric is normalized so that $\mu_{\M}(\mathcal{M})=1$. 
Eigenfunctions of the Laplacian on $\M$ are the key ingredient to carry an analogue of Fourier series on manifolds, the so called Sturm-Liuville's decomposition, here in Proposition \ref{SL}. 
Let $\left\{  \varphi_{k}\right\}  _{k=0}^{\infty}$ be an orthonormal basis of eigenfunctions of the (positive) Laplace-Beltrami operator, with eigenvalues $0=\lambda_{0}^{2}<\lambda_{1}^{2}\leq\lambda_{2}^{2}\leq\ldots$ , $\Delta\varphi_{k}=\lambda_{k}^{2}\varphi_{k}$ and let $\mathcal{M}_{L}=\mathrm{span}\{\varphi_{k}: \lambda_{k}\leq L \}$ be the space of diffusion polynomials of bandwidth $L\geq 0$.
\begin{definition}\label{def:1}
For $N$ points $\{x_{j} \}_{j=1}^{N} \subset \mathcal{M}$ and weights $\{\omega_{j}\}_{j=1}^{N} \subset \mathbb{R}$, we say that $\{(x_{j},\omega_{j}) \}_{j=1}^{N} $ is a cubature of strength $L$ if
\begin{equation}\label{eq:cuba}
\int_{\mathcal{M}} P(x)d\mu_{\M}(x) = \sum_{j=1}^{N} \omega_{j}P(x_{j}) \quad\text{for all }P\in\mathcal{M}_{L}.
\end{equation}
\end{definition}
Since the constant function is contained in $\mathcal{M}_{L}$, \eqref{eq:cuba} implies $\sum_{j=1}^{N} \omega_{j}=1$. 
Moreover, by orthogonality of the eigenfunctions $\varphi_{k}$
\[
\int_{\mathcal{M}} \varphi_{k}(x)d\mu_{\M}(x)=0 \quad \text{for all } k\geq 1,
\]
hence $\{(x_{j},\omega_{j}) \}_{j=1}^{N} $ is a cubature of strength $L$ if and only if
\[
\sum_{j=1}^{N} \omega_{j}P(x_{j})=0 \quad \text{for all } P\in \mathcal{M}_{L}^{0},
\]
where $\mathcal{M}_{L}^{0}=\mathrm{span}\{\varphi_{k}: 0<\lambda_{k}\leq L \}$.

For each $L\geq 0$, we denote by $N(L)$ the minimal number of points in a cubature of strength $L$.

\begin{prop}\label{prop:1}
There exists a positive constant $c_{\mathcal{M}}$ such that if $\{(x_{j},\omega_{j}) \}_{j=1}^{N(L)}$ is a cubature of strength $L$, then  $N(L)\geq c_{\mathcal{M}}L^{d}$ for every $L\geq 0$.
\end{prop}

\proof
The proposition is a coming-on of the results in \cite{BCCGST}.
Let $\{(x_{j},\omega_{j}) \}_{j=1}^{N(L)}$ be a cubature of strength $L$. By  \cite[Theorem 2.12]{BCCGST} there exists a constant $\beta>0$ such that for every function $f$ in the Sobolev space $W^{\alpha,1}(\mathcal{M})$ with $\alpha>d$ one has
\[
\left\vert \int_{\mathcal{M}} f(x)d\mu_{\M}(x) - \sum_{j=1}^{N(L)} \omega_{j}f(x_{j}) \right\vert \leq \beta L^{-\alpha}\| f\|_{W^{\alpha,1}}.
\]
By  \cite[Theorem 2.16]{BCCGST}, there exists also a constant $\gamma>0$ such that for every $L$ there exists a function $f_{L}\in W^{\alpha,1}$ with
\[
\left\vert \int_{\mathcal{M}} f_{L}(x)d\mu_{\M}(x) - \sum_{j=1}^{N(L)} \omega_{j}f_{L}(x_{j}) \right\vert \geq \gamma N(L)^{-\alpha/d}\|f_{L} \|_{W^{\alpha,1}}.
\]
Therefore
\[
\gamma N(L)^{-\alpha/d}\|f_{L} \|_{W^{\alpha,1}}\leq \left\vert \int_{\mathcal{M}} f_{L}(x)d\mu_{\M}(x) - \sum_{j=1}^{N(L)} \omega_{j}f_{L}(x_{j}) \right\vert \leq \beta L^{-\alpha}\| f_{L}\|_{W^{\alpha,1}}
\]
and this gives $N(L)\geq \left(\frac{\gamma}{\beta} \right)^{d/\alpha}L^{d}$. 
\endproof

In \cite{GG}, the authors proved that if $N\geq C_{\mathcal{M}}L^{d}$ with $C_{\mathcal{M}}$ a fixed constant depending only on $\mathcal{M}$, then there exists a set of points $\{x_{j} \}_{j=1}^{N}$ such that $\{(x_{j},1/N) \}_{j=1}^{N}$ is a cubature of strength $L$. What happens if weights are fixed, but not all equal? Existence of a cubature of strength $L$ with the same cardinality in this case is not always guaranteed, as we show in Example \ref{Ex1}.

\begin{ex}\label{Ex1}
This example is a follow up of the theory exposed in \cite{BCCGST}.
Let $\mathcal{M}$ be a $d$-dimensional manifold as described in the beginning of this section. 
Let the weights be given by
\[
\omega_{1}=1-\frac{1}{N+1}, \ \quad \omega_{j}=\frac{1}{(N+1)(N-1)}, \ 2 \leq j \leq  N.
\]
Assume that $\{(x_{j}, \omega_{j}) \}_{j=1}^{N}$ is a cubature of strength $L$ on $\mathcal{M}$. Let $f$ be a function such that $f$ is supported on a ball of radius $cN^{-1/d}$ around $x_{1}$, $f(x_{1})=1$, $\int_{\mathcal{M}}f(x)d\mu_{\M}(x)=N^{-1}$ and 
\[
\|f \|_{W^{d+d\varepsilon,1}}\leq cN^{\frac{d+d\varepsilon}{d}-1}=cN^{\varepsilon},
\]
where $\varepsilon$ is a small positive number and $f$ is a function as in the proof of Theorem 2.16 in \cite{BCCGST}.

By \cite[Theorem 2.12]{BCCGST}, we have
\[
\left\vert \int_{\mathcal{M}} f(x)d\mu_{\M}(x) - \sum_{j=1}^{N(L)} \omega_{j}f(x_{j}) \right\vert \leq C L^{-(d+d\varepsilon)}\| f\|_{W^{d+d\varepsilon,1}}.
\]
Since
\begin{multline*}
\left\vert \int_{\mathcal{M}} f(x)d\mu_{\M}(x) - \sum_{j=1}^{N(L)} \omega_{j}f(x_{j}) \right\vert 
\\
= 
\left\vert \frac{1}{N} - \left(1-\frac{1}{N+1} \right)f(x_{1}) + O\left(N\frac{1}{N^{2}} \right) \right\vert 
= 
1+O\left(\frac{1}{N} \right)
\end{multline*}
and 
\[
L^{-(d+d\varepsilon)}\| f\|_{W^{d+d\varepsilon,1}}\leq L^{-(d+d\varepsilon)}N^{\varepsilon},
\]
one has
\[
N\geq CL^{\frac{d+d\varepsilon}{\varepsilon}}.
\]
Therefore $\{(x_{j}, \omega_{j}) \}_{j=1}^{N}$ cannot be a cubature of strength $L$ with the above choice of weights under the only hypothesis $N\geq CL^{d}$.
\end{ex}
In fact, it has been recently proved, see \cite{BGG}, that the following estimate holds
\[
1\ge C_{\mathcal M} L^d\sum_{j=1}^N \omega_j^2
\]
for all real weights $\{ \omega_{j} \}_{j=1}^{N}$ such that $\sum_{j=1}^{N} \omega_{j}=1$.

\subsection{Cubature formulas on algebraic manifolds}\label{sec:alg and ellipse}
Let $\mathcal{V}\subset \mathbb{R}^n$ be a $d$-dimensional smooth, connected, compact affine algebraic manifold. 
$\mathcal{V}$ carries a Riemannian structure with normalized measure $\mu_{\V}$ induced by the Riemannian metric, such that $\mu_{\V}(\V) = 1$. 
Hence, the definitions and statements of the previous section do apply in principle (provided that $\mathcal{V}$ is orientable). 
Here, however, we replace diffusion polynomials $\mathcal{M}_L$ by the algebraic polynomials on $\R^n$ restricted to $\mathcal{V}$ and denoted by 
\begin{equation*}
\mathcal{V}_L:=\{P_{|\mathcal{V}} : P\in \R[z_1,\ldots,z_n] : \text{ total degree of $P$ is $\leq L$}\}. 
\end{equation*}
For an algebraic manifold and in analogy to Definition \ref{def:1}, we now define the concept of algebraic cubatures.
\begin{definition}\label{def:2}
For $N$ points $\{x_{j} \}_{j=1}^{N} \subset \mathcal{V}$ and weights $\{\omega_{j}\}_{j=1}^{N} \subset \mathbb{R}$, we say that $\{(x_{j},\omega_{j}) \}_{j=1}^{N} $ is an algebraic cubature of strength $L$ if
\[
\int_{\mathcal{V}} P(x)d\mu_{\V}(x) = \sum_{j=1}^{N} \omega_{j}P(x_{j}) \quad\text{for all }P\in\mathcal{V}_L.
\]
\end{definition}
Let $\mathcal{V}_L^0$ denote the orthogonal complement of the constant function within $\mathcal{V}_L$. As before, if $\sum_{j=1}^N \omega_j=1$, then $\{(x_{j},\omega_{j}) \}_{j=1}^{N} $ is an algebraic cubature of strength $L$ if and only if 
\[
\sum_{j=1}^{N} \omega_{j}P(x_{j})=0 \quad \text{for all } P\in \mathcal{V}_{L}^{0}.
\]

The analogue of Proposition \ref{prop:1} also holds for an algebraic manifold, i.e., there exists a positive constant $c_{\mathcal{V}}$ such that if $\{(x_{j},\omega_{j}) \}_{j=1}^{N}$ is an algebraic cubature of strength $L$, then $N\geq c_{\mathcal{V}}L^{d}$ for every $L\geq 0$, the proof of \cite[Proposition 2.1.]{EMOC} for $t$-designs can be extended to this case piece by piece.

Notice that  given a Riemannian algebraic manifold, we have two different definitions for a cubature of strength $L$, the one given in Definition \ref{def:1} and the one given in Definition \ref{def:2}.
Observe though that the first definition is intrinsic, whereas the second is extrinsic and depends on the specific embedding of the manifold in the Euclidean space $\mathbb R^n$.

\subsection{On the relation between polynomials and diffusion polynomials}\label{elipse}

The relation between polynomials and diffusion polynomials on an algebraic manifold has not been deeply understood. 
Nevertheless, we know this relation for some particular manifolds.
In the case of the sphere $\mathbb{S}^{d}$, the eigenfunctions of the Laplacian are polynomials in the ambient space $\mathbb R^{d+1}$ restricted to the sphere. 
The Grassmannian manifold $\mathbb{G}_{k,m}$,
consisting of the $k$-dimensional subspaces of $\mathbb R^m$, can be isometrically embedded into $\mathbb R^{m^2}$ by seeing it as the set of symmetric $m\times m$ matrices 
which are projection operators and have trace equal to $k$ (see \cite[Section 1.3.2.]{chi}).
Any diffusion polynomial on the Grassmannian manifold is then the restriction of a polynomial in the ambient space $\mathbb R^{m^2}$, cf.~\cite{BEG}. 
In general, though, eigenfunctions of the Laplacian on an algebraic manifold $\mathcal V$ are not necessarily restrictions of polynomials. 
In the following example we show that there are diffusion polynomials on the ellipse that are not restrictions of polynomials in the ambient space. Notice that the circle and the ellipse are different algebraic manifolds, but they coincide as Riemannian manifolds.

The (positive) Laplacian $\Delta_{\R/2\pi\mathbb{Z}}$ on $\R/2\pi\mathbb{Z}$ is simply $-\partial_t^2$ acting on $2\pi \mathbb{Z}$ periodic real-valued functions on $\R$. Its eigenvalues are $k^2$, for $k\in\mathbb{N}$, with associated eigenfunctions 
\begin{equation*}
t\mapsto \cos( k t),\qquad t\mapsto \sin( k t). 
\end{equation*}
So we have two eigenfunctions associated to each eigenvalue $k^2$.

\begin{definition}\label{def:3}
For fixed $a,b>0$, we consider the ellipse 
\begin{equation*}
E_{a,b}=\left\lbrace (x,y)\in\R^2 : \frac{x^2}{a^2}+\frac{y^2}{b^2}=1\right\rbrace,
\end{equation*}
which is parametrized by 
\begin{align*}
u_{a,b}:\R/2\pi\mathbb{Z} \rightarrow E_{a,b},\quad t\mapsto (a\cos(t),b\sin(t)).
\end{align*}
\end{definition}

The circle is a particular case of the ellipse, for which the relations of Laplacian eigenfunctions and polynomials are well-studied.
\begin{ex}[Circle $\mathbb{S}^1$]
For $a=b=1$, the mapping $u_{1,1}$ is an arc-length parametrization of $\mathbb{S}^1$, hence, an isometry, so that 
\begin{equation}\label{eq:a=b=1}
\Delta_{\mathbb{S}^1} f = (\Delta_{\R/2\pi\mathbb{Z}}(f\circ u_{1,1}))\circ u^{-1}_{1,1}
\end{equation}
holds for every $f\in C^{\infty}(\mathbb{S}^{1})$, see Proposition \ref{isometry}. 
In particular, the eigenvalues of $\Delta_{\mathbb{S}^1} $ are $k^2$ with associated eigenfunctions 
\begin{align*}
f_k &: \mathbb{S}^1\rightarrow\R, & g_k &: \mathbb{S}^1\rightarrow\R, \\
& (x,y)\mapsto \cos(k u_{1,1}^{-1}(x,y)), & & (x,y)\mapsto \sin(k u_{1,1}^{-1}(x,y)).
\end{align*}
%
Hence, we observe, for $t\in\R/2\pi\mathbb{Z}$,
\begin{align*}
f_k(\cos( t),\sin( t)) & =\cos( k t),\\
g_k(\cos( t),\sin( t)) &=\sin( k t).
\end{align*}
All eigenfunctions of $\Delta_{\mathbb{S}^1} $ are restrictions of algebraic polynomials in $\mathbb{R}^2$, which can be derived from the Chebycheff-polynomials of first and second type, $T_k$ and $U_k$, via
\begin{align*}
\R^2\ni(x,y)&\mapsto T_k(x),& T_k(\cos( t))&=\cos( k t),\\
\R^2\ni(x,y)&\mapsto U_{k-1}(x)y,& U_{k-1}(\cos( t)) \sin( t)& = \sin( k t),\quad t\in\R.
\end{align*}
\end{ex}
We now state that the situation is very different for $a\neq b$.
\begin{prop}\label{p7}
If $a\neq b$, then each nonzero eigenvalue of the Laplacian on $E_{a,b}$ has an eigenfunction that is not the restriction of any algebraic polynomial on $\mathbb{R}^{2}$ with complex coefficients. 
\end{prop}
\begin{proof}
Let us consider the parametrization of the ellipse $u_{a,b}$ given in Definition \ref{def:3}.
For $a\neq b$, $u_{a,b}$ is not an isometry. To compute the arc-length parametrization of $E_{a,b}$, we define $\ell_{a,b}:=\int_{-\pi}^{\pi}\|\dot{u}_{a,b}(t)\|$ and 
\begin{equation}\label{eq:1}
h_{a,b} : [0,2\pi] \rightarrow [0,\ell_{a,b}],\quad t\mapsto  \int_{0}^{t} \|\dot{u}_{a,b}(s)\|\mathrm{d} s.
\end{equation}
We now identify $h_{a,b}$ with its periodic extension $h_{a,b}:\R/2\pi\mathbb{Z}\rightarrow \R/\ell_{a,b}\mathbb{Z}$. 
The arc-length parametrization of $E_{a,b}$ is
\begin{equation*}
\psi_{a,b} : \R/\ell_{a,b}\mathbb{Z} \rightarrow E_{a,b},\quad t\mapsto u_{a,b}(h_{a,b}^{-1}(t))= (a\cos(h_{a,b}^{-1}(t)),b\sin(h_{a,b}^{-1}(t))).
\end{equation*}
We deduce that two linearly independent eigenfunctions on $E_{a,b}$ with respect to the eigenvalue $k^2$, for $0<k\in\mathbb{N}$, are 
\begin{align*}
f_k &: E_{a,b}\rightarrow\R, & g_k &: E_{a,b}\rightarrow\R, \\
& (x,y)\mapsto  \cos\left(\frac{2\pi}{\ell_{a,b}}k\psi_{a,b}^{-1}(x,y)\right), & & (x,y)\mapsto \sin\left(\frac{2\pi}{\ell_{a,b}}k\psi_{a,b}^{-1}(x,y)\right).
\end{align*}
%
They span the eigenspace associated to $k^2$. 
Since $\psi_{a,b}$, $u_{a,b}$, and $h_{a,b}$ are bijections, this implies, for $t\in \R/2\pi\mathbb{Z}$,
\begin{align}
f_k(a\cos(t),b\sin(t))&=\cos\left( \frac{2\pi}{\ell_{a,b}}k h_{a,b}(t)\right),\label{fk 1}\\
g_k(a\cos(t),b\sin(t))&=\sin\left( \frac{2\pi}{\ell_{a,b}}k h_{a,b}(t)\right).\label{gk 1}
\end{align}
To prove our claim, we now assume that both, $f_k$ and $g_k$, are restrictions of algebraic polynomials on $\R^2$, i.e., for $x,y\in E_{a,b}$,
\begin{align*}
f_k(x,y)&= \sum_{m,n\in\mathbb{N}} \alpha_{m,n} x^my^n,\\
g_k(x,y)&= \sum_{m,n\in\mathbb{N}} \beta_{m,n} x^my^n,
\end{align*}
with finitely many nonzero coefficients $\alpha_{m,n},\beta_{m,n}\in\mathbb{C}$. Thus, \eqref{fk 1} and \eqref{gk 1} imply, for $t\in\R/2\pi \mathbb{Z}$, 
\begin{align}
\cos\left( \frac{2\pi}{\ell_{a,b}}k h_{a,b}(t)\right)&=\sum_{m,n\in\mathbb{N}} \alpha_{m,n} a^m\cos^m(t) b^n\sin^n(t),\label{eq:cos 1}\\
\sin\left( \frac{2\pi}{\ell_{a,b}}k h_{a,b}(t)\right) & = \sum_{m,n\in\mathbb{N}} \beta_{m,n} a^m\cos^m(t) b^n\sin^n(t).\label{eq:sin 1}
\end{align}
Trigonometric identities, in particular power reduction formulae and product to sum identities, imply that both, \eqref{eq:cos 1} and \eqref{eq:sin 1}, are trigonometric polynomials, i.e., finite linear combination of $\cos(lt)$ and $\sin(mt)$, $l,m\in\mathbb{N}$. Therefore, their derivatives 
\begin{align}
t & \mapsto -\sin\left( \frac{2\pi}{\ell_{a,b}}k h_{a,b}(t)\right) \frac{2\pi}{\ell_{a,b}}k h_{a,b}'(t),\label{eq:cos 2}\\
t & \mapsto \cos\left( \frac{2\pi}{\ell_{a,b}}k h_{a,b}(t)\right) \frac{2\pi}{\ell_{a,b}}k h_{a,b}'(t), \label{eq:sin 2}
\end{align}
are also trigonometric polynomials. The obvious identity 
\begin{align*}
h_{a,b}'(t) & = \cos^2\left( \frac{2\pi}{\ell_{a,b}}k h_{a,b}(t)\right) h_{a,b}'(t) + \sin^2\left( \frac{2\pi}{\ell_{a,b}}k h_{a,b}(t)\right)h_{a,b}'(t) \\
\end{align*}
implies that $h_{a,b}'$ is a trigonometric polynomial due to \eqref{eq:cos 1}, \eqref{eq:sin 1}, and \eqref{eq:cos 2}, \eqref{eq:sin 2} being trigonometric polynomials and the latter being an algebra. Hence, the definition of $h_{a,b}(t)$ in \eqref{eq:1} yields  that 
\begin{equation*}
h_{a,b}'(t)= \|\dot{u}_{a,b}(t)\|
= \sqrt{a^2\sin^2(t)+b^2\cos^2(t)}
= \sqrt{a^2+b^2+(b^2-a^2)\cos(2t)}
\end{equation*}
is a trigonometric polynomial. The infinite Taylor expansion of the square root implies that the above right-hand-side has an infinite Fourier series if and only if $b^2-a^2\neq 0$. More elementary, since $h_{a,b}'(t)$ is an even function, it must be a finite linear combination of $\cos(lt)$, $l\in\mathbb{N}$. Since the square of $h_{a,b}'(t)$ coincides with $a^2+b^2+(b^2-a^2)\cos(2t)$, the largest $l$ that can occur with nonzero coefficient in $h_{a,b}'$ is $l=1$ due to product to sum identities for the cosine. Since there is no $\cos(t)$ term in the square of $h_{a,b}'(t)$, we deduce $a^2=b^2$, which contradicts the assumption of the proposition. 
\end{proof}

\begin{remark}
Note that in Proposition \ref{p7} we prove a result stronger than needed, since we allow the polynomials to have complex coefficients meanwhile $\V_{L}$ is a vector space of  polynomials with real coefficients.
\end{remark}

\subsection{Notation}\label{sec_Notation}
We use the notation $\gtrsim$ meaning the right-hand side is less or equal to the left-hand side up to a positive constant factor that is only allowed to depend on $\M$ or $\V$ and hence on $d$. 
The symbol $\lesssim$ is used analogously.

For a more compact notation, we make the convention that $\mathbb{X}$ either denotes $\mathcal{M}$ or $\mathcal{V}$ as defined in Section \ref{sec:def}. 
Then $\mu_{\X}$ will denote respectively $\mu_{\M}$ or $\mu_{\V}$,  and $\mathbb{X}_{L}$ and $\mathbb{X}_{L}^{0}$ will do the same with $\mathcal{M}_{L}$ or $\mathcal{V}_{L}$ and $\mathcal{M}_{L}^{0}$ or $\mathcal{V}_{L}^{0}$.

\section{Main result}\label{SecMR}
Our main result holds for the Riemannian manifold $\mathcal{M}$ with diffusion polynomials and for the algebraic manifold $\mathcal{V}$ with algebraic polynomials. 
\begin{Theorem}[Main result]\label{mainM}
Let $h=1$ if $\X=\M$, and $h=d$ if $\X=\V$.
There exists a constant $C = C(\mathcal{\X})$ such that 
for all $0<a\leq 1 \leq b$, if
\[
N\geq C  b \left( \frac{b}{a}\right)^{2h} L^{d}
\]
and if the weights $\{\omega_{j} \}_{j=1}^{N}$ are such that 
\[
\sum_{j=1}^{N} \omega_{j}=1 \ \textit{ and } \ \frac{a}{N}\leq \omega_{j} \leq \frac{b}{N},
\]
then there exists $\{x_{j} \}_{j=1}^{N}\subset \X$ such that $\{(x_{j},\omega_{j}) \}_{j=1}^{N}$ is a  cubature / algebraic cubature of strength $L$.
\end{Theorem}
The proof of Theorem \ref{mainM} is presented in the subsequent section. Here, we prove that the lower bound on the weights can be removed.

\begin{Cor} \label{maincorM}
Let $b\geq 1$. If
\[
N\geq C b^{2h+2}L^{d}
\]
and if the weights $\{\omega_{j} \}_{j=1}^{N}$ are such that 
\[
\sum_{j=1}^{N} \omega_{j}=1 \ \textit{ and } \ 0\leq \omega_{j} \leq \frac{b}{N},
\]
then there exists $\{x_{j} \}_{j=1}^{N}\subset \X$ such that $\{(x_{j},\omega_{j}) \}_{j=1}^{N}$ is a cubature / algebraic cubature of strenght $L$.
\end{Cor}

\begin{proof}
Assume all weights are in increasing order, $\omega_1\leq\omega_2\leq\ldots\leq\omega_N$. 
Let us organize the set of weights in blocks with total mass at least $1/N$. 
Thus let $j_1$ be such that $\sum_{j=1}^{j_1-1}\omega_j< 1/N$ but  $W_1=\sum_{j=1}^{j_1}\omega_j\geq 1/N$. 
Let $j_2$ be such that $\sum_{j=j_1+1}^{j_2-1}\omega_j< 1/N$ but  $W_2=\sum_{j=j_1+1}^{j_2}\omega_j\geq 1/N$, and so on, up until $j_m=N$ in such a way that $\sum_{j=j_{m-1}+1}^{N-1}\omega_j< 1/N$ but  $W_m=\sum_{j=j_{m-1}+1}^{N}\omega_j\geq 1/N$. 
Notice that the construction ends correctly since $\omega_N\geq1/N$. 
By construction, for all $i=1,\ldots,m$, 
\begin{align*}
\frac{1}{N}\leq W_i\leq\frac{b+1}{N},\qquad 1=\sum_{j=1}^{N}\omega_j=\sum_{i=1}^{m}W_i\leq m\frac {b+1}{N},
\end{align*}
so that $m\geq N/{(b+1)}\gtrsim (b+1)^{2h+1}L^d$.
We can therefore apply Theorem \ref{mainM} to the weights $\{W_i\}_{i=1}^{m}$ and we conclude that there are points $\{x_i\}_{i=1}^m$ such that $\{(x_i,W_i)\}_{i=1}^m$ is a cubature of strength $L$. 
By repeating the point $x_i$ for all the weights $\omega_j$ with $j_{i-1}+1\leq j \leq j_i$
we obtain the desired cubature $\{(x_i,\omega_i)\}_{i=1}^{N}$.
\end{proof}

\section{Proof of the Main result} 
\label{SecIn}
As in \cite{BRV, EMOC, GG}, the proof is based on a result from the Brouwer degree theory.

\begin{lemma}\label{brouwer}
\cite[Theorem 1.2.9]{brouwer} Let $H$ be a finite dimensional Hilbert space with inner product $\langle\cdot,\cdot\rangle$. Let $f:H\rightarrow H$
be a continuous mapping and $\Omega$ an open bounded subset with boundary
$\partial\Omega$ such that $0\in\Omega\subset H$. If $\left\langle
x,f\left(  x\right)  \right\rangle >0$ for all $x\in\partial\Omega$, then
there exists $x\in\Omega$ satisfying $f\left(  x\right)  =0.$
\end{lemma}
The following result is the main tool to define a mapping with the properties stated in Lemma \ref{brouwer}.
\begin{lemma}\label{prop_4}
Let $h=1$ if $\X=\M$ and $h=d$ if $\X=\V$. There exists a constant $C=C(\X)>0$ such that the following holds: for all $0<a\leq 1 \leq b$, for all $N\ge Cb(\frac{b}{a})^{2h} L^d$ and for all weights $\{\omega_{j} \}_{j=1}^{N}$ such that 
\[
\sum_{j=1}^{N} \omega_{j}=1 \ \textit{ and } \ \frac{a}{N}\leq \omega_{j} \leq \frac{b}{N},
\]
there exists a continuous mapping
\begin{align*}
F :\X_{L}^{0} &\rightarrow \X^N \\
P & \mapsto (x_{1}(P),\ldots,x_{N}(P)),
\end{align*}
such that for all $P \in \X^0_{L}$ with $\int_{\X} \|\nabla P(x)\|d\mu_{\X}(x)=1$,
\begin{equation*}
 \sum_{j=1}^{N} 
\omega_j P(x_{j}(P))>0.
\end{equation*}
\end{lemma}
We postpone the proof of Lemma \ref{prop_4} and now verify our Main result.

\begin{proof}[Proof of Theorem \ref{mainM}]
Fix $L$ and define
\begin{align}\label{Omega}
\Omega =\left\lbrace P \in \X_{L}^{0} : \int_{\X} \| \nabla P(x)\| d\mu_{\X}(x) 
< 
1 \right\rbrace,
\end{align}
which is clearly an open subset of $\X^0_{L}$ such that $0 \in  \Omega \subset \X_{L}^{0}$.
Since $\int_{\X} \| \nabla P(x)\| d\mu_{\mathbb X} (x)$ is a norm in the finite dimensional space $\X_{L}^{0}$, it is equivalent to the $L^2$ norm in $\X_{L}^{0}$, so $\Omega$ is also bounded in $\X_{L}^{0} \subset L^2 (\X)$.
Take $C = C(\mathcal{\X})$ as in Lemma \ref{prop_4}, let $N\ge Cb(\frac{b}{a})^{2h}L^d$ and let $x_i(P)$ be the points defined by the map $F$ in Lemma \ref{prop_4} for $P\in \partial \Omega$.

By the Riesz Representation Theorem, for each point $x\in \X$ there exists a unique polynomial $G_x \in \X_{L}^{0}$ such that
\begin{equation*}
\left\langle
G_x , P
\right\rangle
=
P(x)
\end{equation*}
for all $P \in \X_{L}^{0}$. Then a set of points $\{ x_j \}_{j=1}^{N} \subset \X$ together with a set of weights $\{ \omega_i \}_{j=1}^{N} \subset \mathbb{R}_+$ is a cubature formula of strength $L$ if and only if 
\begin{equation*}
\sum_{i=1}^{N} 
\omega_{j}
G_{x_{j}}
=
0.
\end{equation*} 
Now let $U: \X^N \rightarrow \X_{L}^{0}$ be the continuous map defined by $U(x_1 , ... ,x_N) = \sum_{i=1}^{N} \omega_{j} G_{x_{j}}$ and let us consider the application
\begin{equation*}
f = U \circ F: \X_{L}^{0} \rightarrow \X_{L}^{0}.
\end{equation*}
Then, by Lemma \ref{prop_4}, for every $P \in \partial\Omega$ we have 
\begin{equation*}
\left\langle
P, f(P)
\right\rangle
=
\sum_{j=1}^{N} 
\omega_{j}
P( x_{j}(P))
>
0.
\end{equation*}
We conclude with Lemma \ref{brouwer}, stating that there exists $Q \in \Omega$ such that $U(F(Q)) = 0$, that is, such that $\sum_{j=1}^{N} \omega_{j} G_{x_{j}(Q)} = 0$, which implies that  $\{(x_{j}(Q),\omega_{j}) \}_{j=1}^{N}$ is a cubature formula of strength $L$. 
\end{proof} 

In order to complete the above proof, we must verify Lemma \ref{prop_4}. 
We define the application $F$ through a gradient flow with initial points that are taken from a partition of the manifold into regions with areas corresponding to the weights. 
To verify suitable properties of the flow, Marcinkiewicz-Zygmund inequalities for the gradient of diffusion polynomials and algebraic polynomials are required. 
These are the topics of the subsequent sections. We start with weighted area partitions, where the result holds for both scenarios and then we prove Marcinkiewicz-Zygmund inequalities separatelly for polynomials and diffusion polynomials. 

\section{Weighted area partitions}\label{SecP1}
Here we generalize the results in \cite{GL} for the case of equal weight partitions to the case of not all equal weights.
\begin{definition}\label{def:partition}
Let $0<a\leq1\leq b$ and $0<c_3<c_4$. We say that a collection of subsets of $\X$, $\mathcal{R}=\{R_{j} \}_{j=1}^{N}$ is a partition of $\X$ with constants $a$, $b$, $c_3$ and $c_4$
if the following hold:
\begin{itemize}
\item $\cup_{j=1}^N R_j=\X$ and $\mu_{\X}(R_i\cap R_j)=0$ for all $1\le i<j\le N$,
\item $a/N\leq\mu_{\X}(R_{j})\leq b/N$ for $j=1, \cdots N$, 
\item each $R_{j}$ is contained in a geodesic ball $X_{j}$ of radius $ c_4 b^{1/d}N^{-1/d} $ and contains a geodesic ball $Y_{j}$ of radius $c_3\left(a^{2}/b \right)^{1/d}N^{-1/d}$.
\end{itemize}
\end{definition}

\begin{prop} \label{partition}
There exist two constants $0<c_3<c_4$ such that for all constants $a$ and $b$ with $0<a\leq 1\leq b$, for every $N \geq1$ and for every choice of weights $\{\omega_{j} \}_{j=1}^{N}$ with $\sum_{j=1}^{N} \omega_{j}=1$ and $a/N \leq \omega_{j} \leq b/N$,  
there is a partition of $\X$, $\mathcal{R}=\{R_{j} \}_{j=1}^{N}$, with constants $a$, $b$, $c_3$ and $c_4$ such that $\mu_{\X}(R_{j})=\omega_j$ for all $j=1,\ldots,N$.

\end{prop}

The proof of Proposition \ref{partition} is based on the following lemma on non-atomic measures not having gaps in their range:
\begin{lemma}\label{lemma:gg}
\cite[Corollary 3]{GL}
Let $S$ be a measurable subset of $\X$. Then, for any $0\leq r\leq \mu_{\X}(S)$, there is $\Gamma\subset S$ such that $\mu_{\X}(\Gamma)=r$. 
\end{lemma}
Note that this lemma holds for more general spaces $\X$ than the ones we consider in the present manuscript, see \cite{GL} for a brief discussion. 

\begin{Cor}\label{cor:010}
Given positive weights $\{\omega_j\}_{j=1}^N\subset\mathbb{R}$, let $S$ and $Q_1,\ldots,Q_N\subset S$ be measurable subsets of $\X$. If $\{Q_j\}_{j=1}^N$ are pairwise disjoint with $\mu_{\X}(Q_j)\leq \omega_j$ and $\mu_{\X}(S)\geq \sum_{j=1}^N \omega_j$, then there are pairwise disjoint $R_1,\ldots,R_N\subset S$, such that $Q_j\subset R_j$ and $\mu_{\X}(R_j)=\omega_j$, $j=1,\ldots,N$.
\end{Cor}
\begin{proof}[Proof of Corollary \ref{cor:010}]
We start with $S_1:=S\setminus \bigcup_{j=1}^N Q_j$. Since $\mu_{\X}(S_1)\geq \omega_1-\mu_{\X}(Q_1)$, there is $\Gamma_1\subset S_1$ such that $\mu_{\X}(\Gamma_1)=\omega_1-\mu_{\X}(Q_1)$. We set $R_1:=Q_1\cup \Gamma_1$. Next, we define $S_2:=S_1\setminus R_1$. There is $\Gamma_2\subset S_2$ such that $\mu_{\X}(\Gamma_2)=\omega_2-\mu_{\X}(Q_2)$. Let $R_2:=Q_2\cup \Gamma_2$ and so on.  
\end{proof}
Our assumptions on $\X$ imply that there are constants $0<c_1\leq c_2<\infty$ such that 
\begin{equation}\label{eq:doubling}
c_1 r^d \leq \mu_{\X}(B(x,r)) \leq c_2 r^d,\quad \text{ for all } x\in\X,\; 0<r\leq \diam(\X),
\end{equation}
where $B(x,r)$ denotes the ball of radius $r$ centered at $x$. 
The proof of Proposition \ref{partition} proceeds as in the equal weight case in \cite{GL}, with a few technical
modifications. As in \cite{GL}, we know that there is a family of $\delta$-adic cubes in $\X$, i.e., for any $0<\delta<1$ there exist $0<u_1\leq u_2<\infty$, a collection of open subsets $\{Q^k_\alpha: k\in\mathbb{Z},\; \alpha\in I_k\}$ in $\X$, where each $I_k$ is a finite index set, and points $\{z^k_\alpha: k\in\mathbb{Z},\; \alpha\in I_k\}$ with  
\begin{itemize}
\item[i)] $\mu_{\X}(\X\setminus \bigcup_{\alpha\in I_k} Q^k_\alpha)=0$, for all $k\in\mathbb{Z}$,
\item[ii)] for $l>k$ and $\alpha\in I_l$, there is $\beta_0\in I_k$ such that 
\begin{itemize}
\item[$\bullet$] $Q^l_\alpha\subset Q^k_{\beta_0}$,
\item[$\bullet$] $Q^l_\alpha\cap Q^k_\beta=\emptyset$, for all $\beta\in I_k$ with $\beta\neq \beta_0$
\end{itemize}
\item[iii)] $B(z^k_\alpha,u_1\delta^k)\subset Q^k_\alpha\subset B(z^k_\alpha,u_2\delta^k)$, for all $k\in\mathbb{Z}$, $ \alpha\in I_k$.
\end{itemize}
Assume first
\begin{equation}\label{eq:cond N}
N\geq\frac{2b}{c_1\delta^d\diam(\X)^d}.
\end{equation}
Choose $k\in\mathbb{Z}$ such that 
\begin{equation}\label{eq:est k etc}
u_1\delta^{k+1}< \left( \frac{2}{c_1}\frac{b}{N}\right)^{1/d}\leq u_1\delta^k,
\end{equation}
so that we obtain the estimates
\begin{equation}\label{eq:est 34}
\mu_{\X}(Q^k_\alpha) \geq \mu_{\X}(B(z^k_\alpha, u_1\delta^k)) \geq c_1u_1^d\delta^{kd}\geq 2\frac{b}{N}.
\end{equation}
Here, we have used \eqref{eq:doubling}, so that we still need to ensure $u_1\delta^k\leq \diam(\X)$. Indeed, using \eqref{eq:cond N} we derive
\begin{align*}
u_1\delta^k & = \frac{u_1\delta^{k+1}}{\delta} \leq \frac{(\frac{2}{c_1}\frac{b}{N})^{1/d}}{\delta}\leq \diam(\X).
\end{align*} 
Thus, \eqref{eq:est 34} is a valid estimate. 

Similarly, we derive an upper bound
\begin{align*}
\mu_{\X}(Q^k_\alpha) \leq \mu_{\X}(B(z^k_\alpha, u_2\delta^k)) &\leq c_2u^d_2\delta^{kd} < \frac{c_2}{c_1}\Big(\frac{u_2}{u_1}\Big)^d\frac{2}{\delta^{d}}  \frac{b}{N} =\frac{c_2}{c_1}\Big(\frac{u_2}{u_1}\Big)^d\frac{2}{\delta^{d}}  \frac{b}{a}\frac{a}{N}.
\end{align*}
With $C:=\frac{c_2}{c_1}\Big(\frac{u_2}{u_1}\Big)^d\frac{2}{\delta^d}3^d  \frac{b}{a}$, we have checked
\begin{equation}\label{eq:vol k}
2\frac{b}{N}\leq \mu_{\X}(Q^k_\alpha) \leq \frac{C}{3^d} \frac{a}{N}.
\end{equation}

For the cube generation $k$, we now build a graph with vertices $I_k$. For $\alpha,\beta\in I_k$, we put an edge $(\alpha,\beta)$ if and only if $B(z^k_\alpha,u_1\delta^k)\cap B(z^k_\beta,u_1\delta^k)\neq \emptyset$. 
This graph is connected, see \cite[Proof of Theorem 2]{GL}, so that we can extract a spanning tree with leaf nodes, intermediate nodes, and one root node. 
We create the directed tree $\mathcal{T}$ by directing the edges from the root towards the leaves, so that $(\alpha,\beta)\in\mathcal{T}$ is the directed edge between $\alpha$ and its child $\beta$. 

The triangular inequality yields 
\begin{equation*}
Q^k_\alpha\cup \!\!\!\bigcup_{(\alpha,\beta)\in \mathcal{T}}Q^k_\beta \;\;\;\subset\;\;B(z^k_\alpha,3u_2\delta^k),
\end{equation*}
cf.~\cite[Corollary 2]{GL}. Hence, we obtain the volume estimate
\begin{equation}\label{eq:trian}
\mu_{\X}(Q^k_\alpha\cup \bigcup_{(\alpha,\beta)\in \mathcal{T}}Q^k_\beta) \leq \mu_{\X}(B(z^k_\alpha,3u_2\delta^k))\leq c_2(3u_2\delta^k)^d\leq C\frac{a}{N}.
\end{equation}
We now aim to take a younger generation of $\delta$-adic cubes, say $l=k+m$, such that all cubes of
generation $l$ have measure smaller than $\frac{1}{C}\frac{a}{N}$. Indeed, 
let $m$ be the positive integer such that 
\begin{equation}\label{eq:def m}
\delta^{m}\leq 3 C^{-2/d}<\delta^{m-1}.
\end{equation}
Notice that $3 C^{-2/d}<1$, so that this choice is possible. 
Thus, for all $\alpha\in I_{l}$, we get from \eqref{eq:def m} and \eqref{eq:est k etc}
\begin{align*}
\mu_{\X}(  Q_{\alpha}^{l})  \leq\mu_{\X}(  B(z^l_\alpha, u_{2}\delta^{l}))  &\leq (  c_{2}u_{2}^{d}\delta
^{kd})  \delta^{md} \\
& \leq (  c_{2}u_{2}^{d}\delta
^{kd})3^d C^{-2}\\
&\leq \frac{Ca}{3^dN} 3^dC^{-2} \leq \frac{1}{C}\frac{a}{N}.
\end{align*}
We now construct the partition by running through the directed tree $\mathcal{T}$ and using the above estimates, which are overkill for the leaves but are more appropriate for the remaining nodes. Let us denote the weights by $\Omega:=\{\omega_j\}_{j=1}^N$.

\subsubsection*{Leaves}
Start with a leaf node $\alpha\in I_k$. Take a maximal set of weights from $\Omega$ such that their sum is not bigger than $\mu_{\X}(Q^k_\alpha)$. Denote this maximal set with $\Omega_\alpha$ and its cardinality with $N_\alpha$. Each cube of generation $l$ has measure at most $\frac{1}{C}\frac{a}{N}$, so that the volume of $N_\alpha$ cubes of generation $l$ is bounded by
\begin{align*}
N_\alpha\frac{1}{C}\frac{a}{N} \leq \frac{1}{C}\mu_{\X}(Q^k_\alpha) &\leq \frac{1}{C}\mu_{\X}\left(Q^k_\alpha\cup \bigcup_{(\alpha,\beta)\in T}Q^k_\beta\right)
\leq \frac{a}{N},
\end{align*}
where we have used \eqref{eq:trian}. According to \eqref{eq:vol k}, $Q^k_\alpha$ has sufficient volume that we can choose $N_\alpha$ cubes of generation $l$ inside of $Q^k_\alpha$. Let us denote them by $Q^l_{\beta_1},\ldots,Q^l_{\beta_{N_\alpha}}$. By Corollary \ref{cor:010}, we enlarge each of such cubes within $Q^k_\alpha$, so that their measure matches the weights in $\Omega_\alpha$, so that we obtain $\{R_{\beta_i}\}_{i=1}^{N_\alpha}$. The remainder in $Q^k_\alpha$, i.e., $W_\alpha:=Q^k_\alpha\setminus \bigcup_{i=1}^{N_\alpha} R_{\beta_i}$ has volume less than $b/N$, because we took the maximal number of weights. 

We repeat the above steps for each leaf node but only allow weights in $\Omega$ that have not been chosen previously. After having finished all leaves, we have remainders $W_\alpha\subset Q^k_\alpha$, for each $\alpha\in I_k$ that corresponds to a leaf. 

\subsubsection*{Intermediate nodes}
For each $\alpha\in I_k$ that is neither a leaf nor the root, start with $X_{\alpha}=Q_{\alpha}^{k}\cup\bigcup_{(  \alpha
,\beta)  \in \mathcal{T}}W_{\beta}$, that is we add all the remainders coming
from the children of $\alpha$. Note that we can proceed with the intermediate nodes in an ordering such that the remainders $W_{\beta}$ with $(  \alpha
,\beta)$ have indeed all been already computed. Note also that we can assume $W_{\beta}\subset Q_{\beta}^{k}$, for all $(
\alpha,\beta)  \in \mathcal{T}$ (take this as an induction hypothesis. It is clearly true if $\beta$ is a leaf node, and will follow at the end of this paragraph for the intermediate nodes). Now
repeat the same argument as before with $X_{\alpha}$ in place of $Q^k_\alpha$. 
Take a maximal set of the remaining weights from $\Omega$ such that their sum is not bigger than $\mu_{\X}(X^k_\alpha)$. Again, denote this maximal set with $\Omega_\alpha$ and its cardinality with $N_\alpha$. 
As we saw before, the entire volume of $N_\alpha$ cubes of generation $l$ is at most $a/N$, so that they can be chosen within $Q_{\alpha}^{k}$. Let us denote these cubes by $Q^l_{\beta_1},\ldots,Q^l_{\beta_{N_\alpha}}$. The volume of $Q_{\alpha}^{k}\setminus \big(\bigcup_{i=1}^{N_\alpha} Q^l_{\beta_i}  \big)$ is still at least $b/N$. According to Lemma \ref{lemma:gg}, there is $W_\alpha\subset Q^k_\alpha\setminus \big(\bigcup_{i=1}^{N_\alpha} Q^l_{\beta_i}  \big)$ with volume 
\begin{equation*}
\mu_{\X}(W_\alpha)=\mu_{\X}(X_\alpha)-\sum_{\omega\in\Omega_\alpha}\omega<b/N.
\end{equation*}
By Corollary \ref{cor:010}, we extend the cubes $Q^l_{\beta_1},\ldots,Q^l_{\beta_{N_\alpha}}$ within $X_\alpha\setminus W_\alpha$, so that the volumes match the weights in $\Omega_\alpha$, yielding subsets $\{R_{\beta_i}\}_{i=1}^{N_\alpha}$. By comparing volumes, the union of the extensions now covers the neighboring remainders $W_\beta$ (at least up to a set of measure zero), and the new remainder $W_\alpha$ is indeed contained in $Q^k_\alpha$.  

We proceed with the remaining weights for each of the intermediate nodes in a suitable order. 

\subsubsection*{Root}
We do the same as for intermediate nodes but comparing volumes yields that the remainder of the root node must have measure zero. 

\bigskip
After having treated each node in $\mathcal{T}$, we have collected a partition $\{R_j\}_{j=1}^N$, so that we obtain, with a suitable reordering, $\mu_{\X}(R_j)=\omega_j$, for $j=1,\ldots,N$. 

Since each $R_j$ contains a cube of generation $l$, it contains a ball of radius $u_1\delta^l$. A short calculation yields $\delta^l \gtrsim \big(\frac{a^2}{b}\big)^{1/d} N^{-1/d}$. On the other hand, each $R_j$ is contained in a ball of radius $3u_2\delta^k\lesssim b^{1/d}N^{-1/d}$.

Assume now that 
\[
1\leq N\leq \frac{2b}{c_1\delta^d\diam(\X)^d}.
\]
Let $k$ be now the unique integer such that 
\[
c_2(u_2\delta^{k})^d\leq \frac aN <c_2(u_2\delta^{k-1})^d.
\]
This implies that all the cubes of generation $k$ have measure smaller than all the values $\omega_j$. Then take any $N$ distinct cubes of generation $k$ and extend them by means of Corollary \ref{cor:010} to disjoint sets $\{R_j\}_{j=1}^N$ with measures $\omega_j$, respectively. Each $R_j$ contains its corresponding cube of generation $k$ and therefore a ball with radius 
\[
u_1\delta^k>\frac {u_1\delta}{u_2c_2^{1/d}}\frac{a^{1/d}}{N^{1/d}}.
\]
On the other hand, every $R_j$ is trivially contained in a (any) ball with radius 
\[
\diam{\X}\le\frac{2^{1/d}}{c_1^{1/d}\delta}\frac{b^{1/d}}{N^{1/d}}.
\]
This concludes the proof of Proposition \ref{partition}.

\begin{remark}
Proposition \ref{partition} and its proof hold for any complete, connected metric measure spaces that satisfy \eqref{eq:doubling}, cf.~\cite{GL} for further details. 
\end{remark}

\section{Marcinkiewicz-Zygmund inequalities}\label{Sec6}
\subsection{MZ-inequalities on Riemannian manifolds}\label{SSec61}
The Marcinkiewicz-Zygmund inequality for diffusion polynomials on manifolds has been proved by Maggioni, Mhaskar and Filbir throughout the papers \cite{FM2010, FM2011, MM}. 
For our proof we need the Marcinkiewicz-Zygmund inequality for the gradient of diffusion polynomials. 
Note that when $\mathcal{M}$ is the $d$-dimensional sphere, then the gradient of a polynomial is again a polynomial. 
In the case of a general Riemannian manifold, this fails (see \cite{GG}). Here we prove a Marcinkiewicz-Zygmund inequality for gradients of diffusion polynomials in the case of a Riemannian manifold and with prefixed weights.
Throughout this section, let $\M$, $\mu_{\M}$, $\M_{L}$ and $\M_{L}^{0}$ be as defined in Section \ref{sec:M}.

\begin{prop} \label{MZM}
Let $\mathcal{M}$ be as in Section \ref{sec:M} and
let $0<c_3<c_4$.
Then, there exists a constant $C=C(\M,c_3,c_4)\geq1$ such that for all $0<a\leq1\leq b$, for all integers $N\ge C b(\frac{b}{a})^{2} L^d$, for all partitions $\{ R_{j}\}_{j=1}^{N}$ with constants $a$, $b$, $c_3$ and $c_4$ as in Definition \ref{def:partition},  for all $x_{j} \in R_{j}$, for all $P \in \M_{L}^{0}$ it holds
\begin{align}\label{eq:1/2 estimate}
\Big|\int_{\M} \|\nabla P(x)\| d\mu_{\M}(x)  -  \sum_{j=1}^N \omega_j \|\nabla P(x_j)\| \Big| \leq \frac{1}{2}\int_{\M} \|\nabla P(x)\| d\mu_{\M}(x),
\end{align}
where $\omega_j=\mu_{\M}(R_j)$, for all $j=1,\ldots,N$.
\end{prop} 
\begin{proof}


This proof follows the sketch of the proof of \cite[Theorem 5]{GG}. Fix $\varepsilon>0$ and let $v_{\varepsilon}:[0,+\infty]\to \mathbb{R}$ be a $\mathcal{C}^{\infty}$ function such that
\begin{equation}\label{eq:v}
v_{\varepsilon}(u)=\begin{cases}
u & \mathrm{if } \ u\geq \varepsilon\\
\varepsilon/2 & \mathrm{if } \ u\leq \varepsilon/4
\end{cases}
\end{equation}
and $v_{\varepsilon}(u)\geq u$ for all $u\geq 0$. Let $P \in \M_{L}^{0}$ and let $T$ and $S$ be the vector fields defined as
\[
T(x)=\frac{\nabla P(x)}{v_{\varepsilon}\left(\|\nabla P(x) \| \right)}, \ \quad \ S(x)= \frac{\nabla TP(x)}{v_{\varepsilon}\left(\|\nabla TP(x) \| \right)}.
\]
Therefore, by (\ref{grad})
\begin{align}
TP\left(  x\right) & =\left\langle \nabla P\left(  x\right)  ,\frac{\nabla
P\left(  x\right)  }{v_{\varepsilon}\left(  \left\Vert \nabla P\left(
x\right)  \right\Vert \right)  }\right\rangle\\
STP\left(  x\right) & =\left\langle \nabla TP\left(  x\right)  ,\frac{\nabla
TP\left(  x\right)  }{v_{\varepsilon}\left(  \left\Vert \nabla TP\left(
x\right)  \right\Vert \right)  }\right\rangle.
\label{eq:TP}
\end{align}
We define also for every $L\geq 0$ the kernel $W_{L}$ as
\begin{equation}\label{kernelW}
W_{L}(x,y)=\sum_{\lambda_{k}>0} \frac{1}{\lambda_{k}^{2}}H\left(\frac{\lambda_{k}}{L} \right)\varphi_{k}(x)\varphi_{k}(y),
\end{equation}
where $H$ is a $\mathcal{C}^{\infty}$ even function such that
\[
H(u)=\begin{cases}
1 & \mathrm{if } \ u\in[-1,1]\\
0 & \mathrm{if } \ |u|\geq 2.
\end{cases}
\]
Let  $\Psi_{L}\left(  x,y\right)  $ be a reproducing kernel for $\M_{L}^{0}$ defined as
\begin{equation}\label{repkernel}
\Psi_{L}\left(  x,y\right)     =\Delta_{y}W_{L}\left(  x,y\right)
=\sum_{0<\lambda_{k}}H\left(  \frac{\lambda_{k}}{L}\right)  \varphi
_{k}\left(  x\right)  \varphi_{k}\left(  y\right),
\end{equation}
where with the sub-index we put in evidence with respect to which variable each vector field operates.

It has been proved in \cite[proof of Theorem 5]{GG} that 
there exists a constant $\kappa>0$ such that
\begin{equation}
\label{estimate}
\|\nabla_{y}S_{x}T_{x}W_{L}(x,y) \|\leq \kappa L^{d+1}\left(1+L|x-y|\right)^{-d-1}.
\end{equation}

Notice that
\[
TP\left(  x\right)  =\left\langle \nabla P\left(  x\right)  ,\frac{\nabla
P\left(  x\right)  }{v_{\varepsilon}\left(  \left\Vert \nabla P\left(
x\right)  \right\Vert \right)  }\right\rangle =\frac{\left\Vert \nabla
P\left(  x\right)  \right\Vert ^{2}}{v_{\varepsilon}\left(  \left\Vert \nabla
P\left(  x\right)  \right\Vert \right)  }\leq\left\Vert \nabla P\left(
x\right)  \right\Vert
\]
and therefore
\begin{align*}
&  \left\vert \int_{\mathcal{M}}\left\Vert \nabla P\left(  x\right)
\right\Vert d\mu_{\M}\left(  x\right)  -\sum_{j=1}^{N}\omega_{j}\left\Vert \nabla
P\left(  x_{j}\right)  \right\Vert \right\vert \\
&  \leq\left\vert \int_{\mathcal{M}}\left(  \left\Vert \nabla P\left(
x\right)  \right\Vert -TP\left(  x\right)  \right)  d\mu_{\M}\left(  x\right)
\right\vert +\left\vert \int_{\mathcal{M}}TP\left(  x\right)  d\mu_{\M}\left(
x\right)  -\sum_{j=1}^{N}\omega_{j}TP\left(  x_{j}\right)  \right\vert \\
&  +\left\vert \sum_{j=1}^{N}\omega_{j}\left(  TP\left(  x_{j}\right)
-\left\Vert \nabla P\left(  x_{j}\right)  \right\Vert \right)  \right\vert \\
&  \leq2\varepsilon+\left\vert \int_{\mathcal{M}}TP\left(  x\right)
d\mu_{\M}\left(  x\right)  -\sum_{j=1}^{N}\omega_{j}TP\left(  x_{j}\right)
\right\vert .
\end{align*}
Let $\delta$ be the maximum diameter of the balls $X_{j}$ as in Definition \ref{def:partition}, so  $\delta\le 2c_4b^{1/d}N^{-1/d}.$ Hence, also by (\ref{geodesic}),
\begin{align*}
\left\vert \int_{\mathcal{M}}TP\left(  x\right)  d\mu_{\M}\left(  x\right)
-\sum_{j=1}^{N}\omega_{j}TP\left(  x_{j}\right)  \right\vert  &  \leq
\sum_{j=1}^{N}\int_{R_{j}}\left\vert TP\left(  x\right)  -TP\left(
x_{j}\right)  \right\vert d\mu_{\M}\left(  x\right)  \\
&  \leq\sum_{j=1}^{N}\omega_{j}\sup_{x,z\in R_{j}}\left\vert TP\left(
x\right)  -TP\left(  z\right)  \right\vert \\
& \leq \sum_{j=1}^{N}\omega_{j}\sup_{x,z\in R_j}\sup_{t\in [0,|x-z|]}\left\Vert \nabla
TP\left(  \alpha(t)\right)  \right\Vert |x-z|.
\end{align*}
where $\alpha$ is a normalized geodesic joining $x$ and $z$. 
Since $R_j$ is contained in the ball $X_j$, the geodesic $\alpha$ is contained in the
ball $2X_j$ with the same center as $X_j$ and radius twice the radius of $X_j$.
It follows that 
\[
\left\vert \int_{\mathcal{M}}TP\left(  x\right)  d\mu_{\M}\left(  x\right)
-\sum_{j=1}^{N}\omega_{j}TP\left(  x_{j}\right)  \right\vert  
\leq
\delta\sum_{j=1}^{N}\omega_{j}\sup_{x\in 2X_{j}}\left\Vert \nabla
TP\left(  x\right)  \right\Vert .
\]
From equation \eqref{eq:TP} one has that
\[
STP\left(  x\right)=\frac{\|\nabla TP(x)\|^2}{v_{\varepsilon}(\|\nabla TP(x)\|)} \leq\left\Vert \nabla TP\left(  x\right)  \right\Vert
\]
and therefore%
\begin{align*}
&\delta\sum_{j=1}^{N}\omega_{j}\sup_{x\in 2X_{j}}\left\Vert \nabla TP\left(
x\right)  \right\Vert \\
 &  \leq\delta\sum_{j=1}^{N}\omega_{j}\sup_{x\in 2X_{j}%
}\left\vert \left\Vert \nabla TP\left(  x\right)  \right\Vert -STP\left(
x\right)  \right\vert +\delta\sum_{j=1}^{N}\omega_{j}\sup_{x\in 2X_{j}%
}\left\vert STP\left(  x\right)  \right\vert \\
&  \leq\delta\varepsilon+\delta\sum_{j=1}^{N}\omega_{j}\sup_{x\in 2X_{j}%
}\left\vert STP\left(  x\right)  \right\vert .
\end{align*}
Hence we have obtained 
\[
\left\vert \int_{\mathcal{M}}\left\Vert \nabla P\left(  x\right)  \right\Vert
d\mu_{\M}\left(  x\right)  -\sum_{j=1}^{N}\omega_{j}\left\Vert \nabla P\left(
x_{j}\right)  \right\Vert \right\vert \leq\left(  2+\delta\right)
\varepsilon+\delta\sum_{j=1}^{N}\omega_{j}\sup_{x\in 2X_{j}}\left\vert
STP\left(  x\right)  \right\vert .
\]
We need to estimate the sum above. Notice that by Green's formula (\ref{Green}) we have
\begin{align*}
P\left(  x\right)   &  =\int_{\mathcal{M}}P\left(  y\right)  \Psi_{L}\left(
x,y\right)  d\mu_{\M}\left(  y\right)  =\int_{\mathcal{M}}P\left(  y\right)
\Delta_{y}W_{L}\left(  x,y\right)  d\mu_{\M}\left(  y\right)  \\
&  =\int_{\mathcal{M}}\left\langle \nabla_{y}P\left(  y\right)  ,\nabla
_{y}W_{L}\left(  x,y\right)  \right\rangle d\mu_{\M}\left(  y\right),
\end{align*}
where $W_L$ is defined in (\ref{kernelW}) and $\Psi_{L}$ in (\ref{repkernel}).
Therefore by \eqref{estimate} we have
\begin{align*}
&  \delta\sum_{j=1}^{N}\omega_{j}\sup_{x\in 2X_{j}}\left\vert STP\left(
x\right)  \right\vert \\
& =  \delta\sum_{j=1}^{N}\omega_{j}\sup_{x\in 2X_{j}}\left\vert \int%
_{\mathcal{M}}\left\langle \nabla_{y}P\left(  y\right)  ,\nabla_{y}S_{x}%
T_{x}W_{L}\left(  x,y\right)  \right\rangle d\mu_{\M}\left(  y\right)  \right\vert
\\
&  \leq\delta\sum_{j=1}^{N}\omega_{j}\sup_{x\in 2X_{j}}\int_{\mathcal{M}%
}\left\Vert \nabla_{y}P\left(  y\right)  \right\Vert \left\Vert \nabla
_{y}S_{x}T_{x}W_{L}\left(  x,y\right)  \right\Vert d\mu_{\M}\left(  y\right)  \\
&  \leq\kappa\delta\sum_{j=1}^{N}\omega_{j}\sup_{x\in 2X_{j}}\int%
_{\mathcal{M}}\left\Vert \nabla_{y}P\left(  y\right)  \right\Vert
L^{d+1}\left(  1+L\left\vert x-y\right\vert \right)  ^{-d-1}d\mu_{\M}\left(
y\right)  \\
&  \leq\kappa\delta\int_{\mathcal{M}}\left\Vert \nabla_{y}P\left(  y\right)
\right\Vert \left(  \sum_{j=1}^{N}L^{d+1}\omega_{j}\sup_{x\in 2X_{j}}\left(
1+L\left\vert x-y\right\vert \right)  ^{-d-1}\right)  d\mu_{\M}\left(  y\right)  .
\end{align*}
We reduce to estimate the sum in the integral. To do this, for any fixed $y$, let $J=\left\{  j:\mathrm{dist}\left(  2X_{j},y\right)
\geq2\delta\right\}  $ and $J^{\prime}$ its complement. We start considering $j \in J$. If we call
$q_{j}$ the point in $2X_{j}$ closest to $y$, and $p_{j}$ the
point in $2X_{j}$ farthest from $y,$ then 
\[
1+\frac{L}{2}\left\vert p_{j}-y\right\vert \leq1+\frac{L}{2}\left(  \left\vert
q_{j}-y\right\vert +2\delta\right)  \leq1+L\left\vert q_{j}-y\right\vert
\]
and therefore for the sum over $J$, also by (\ref{radiale}), we have
\begin{align*}
&  \sum_{j\in J}L^{d+1}\omega_{j}\sup_{x\in 2X_{j}}\left(  1+L\left\vert
x-y\right\vert \right)  ^{-d-1}\\
&  =\sum_{j\in J}L^{d+1}\omega_{j}\left(  1+L\left\vert q_{j}-y\right\vert
\right)  ^{-d-1}\leq\sum_{j\in J}L^{d+1}\omega_{j}\left(  1+\frac{L}%
{2}\left\vert p_{j}-y\right\vert \right)  ^{-d-1}\\
&  =\sum_{j\in J}\int_{R_{j}}L^{d+1}\left(  1+\frac{L}{2}\left\vert
p_{j}-y\right\vert \right)  ^{-d-1}d\mu_{\M}\left(  x\right)  \\
&  \leq\sum_{j\in J}\int_{R_{j}}L^{d+1}\left(  1+\frac{L}{2}\left\vert
x-y\right\vert \right)  ^{-d-1}d\mu_{\M}\left(  x\right)  \\
&  \leq\int_{\mathcal{M}  }L^{d+1}\left(
1+\frac{L}{2}\left\vert x-y\right\vert \right)  ^{-d-1}d\mu_{\M}\left(  x\right)
\\
&  \leq c_5 L^{d+1}\int_{0}^{+\infty}\left(  1+\frac{L}{2}s\right)
^{-d-1}s^{d-1}ds\\
&  \leq c_5 L^{d+1}\left(  \int%
_{0}^{1/L}s^{d-1}ds+\left(  \frac{2}{L}\right)  ^{d+1}\int%
_{1/L}^{+\infty}s^{-2}ds\right)  \\
&  \leq (d^{-1}+2^{d+1}) c_{5}L.
\end{align*}
Now we consider $J^{\prime}$. We have that its cardinality is bounded above by the number of
inner balls $Y_j$ that are contained in the ball $B(y, 4\delta)$, and this number is 
bounded above by the ratio
\[
\frac{\mu_{\M}(B(y,4\delta))}{\min_{j=1,\ldots, N}\mu_{\M}(Y_j)}\le\frac{8^d c_2 c_4^db^{2}}{c_1c_3^d a^{2}}.
\]
Thus, 
since $\frac{a}{N}\leq \omega_{j}\leq \frac{b}{N}$, and assuming $N\geq b^3a^{-2}L^d$,
\begin{align*}
\sum_{j\in J^{\prime}}L^{d+1}\omega_{j}\sup_{x\in 2X_{j}}\left(  1+L\left\vert
x-y\right\vert \right)  ^{-d-1}&\leq\sum_{j\in J^{\prime}}L^{d+1}\omega_{j}
\leq
\frac{8^dc_2c_4^d b^{3}}{c_1 c_3^da^{2}}\frac{L^{d+1}}N\leq
\frac{8^dc_2c_4^d}{c_1 c_3^d}L.
\end{align*}
We have obtained
\begin{align*}
&\left\vert \int_{\mathcal{M}}\left\Vert \nabla P\left(  x\right)  \right\Vert
d\mu_{\M}\left(  x\right)  -\sum_{j=1}^{N}\omega_{j}\left\Vert \nabla P\left(
x_{j}\right)  \right\Vert \right\vert \\
\leq&\left(  2+\delta\right)
\varepsilon+\kappa\left((d^{-1}+2^{d+1})c_5+\frac{8^dc_2c_4^d }{c_1 c_3^d}\right)\delta L\int_{\mathcal{M}}\left\Vert \nabla P\left(  y\right)
\right\Vert d\mu_{\M}\left(  y\right)  .
\end{align*}
If we take
\[
\varepsilon=\frac{\kappa\left((d^{-1}+2^{d+1})c_5+\dfrac{8^dc_2c_4^d }{c_1 c_3^d}\right)\delta L}{2+\delta}\int_{\mathcal{M}}\left\Vert \nabla P\left(
y\right)  \right\Vert d\mu_{\M}\left(  y\right),
\]
we obtain%
\[
\left\vert \int_{\mathcal{M}}\left\Vert \nabla P\left(  x\right)  \right\Vert
d\mu_{\M}\left(  x\right)  -\sum_{j=1}^{N}\omega_{j}\left\Vert \nabla P\left(
x_{j}\right)  \right\Vert \right\vert \leq \tilde{C}b^{1/d}LN^{-1/d}\int_{\mathcal{M}%
}\left\Vert \nabla P\left(  y\right)  \right\Vert d\mu_{\M}\left(  y\right),
\]
where
\begin{equation}\label{eq:CC}
\tilde{C}=4c_4\kappa\left((d^{-1}+2^{d+1})c_5+\frac{8^dc_2 c_4^d}{c_1c_3^d}\right).
\end{equation}
Assuming now $N\geq 2^d\tilde{C}^dbL^d$ one obtains
\eqref{eq:1/2 estimate}.  The Proposition now follows with 
$C=\max\{1,2^d\tilde{C}^d\}$.
\end{proof}

\subsection{MZ-inequalities on algebraic manifolds}
For an algebraic manifold with fixed weights, the Marcinkiewicz-Zygmund inequality does not follow from \cite{FM2010, FM2011, MM} as in Section \ref{SSec61} since we are not dealing with diffusion polynomials but with algebraic polynomials. 
Instead, one can apply arguments that use complexification of the variety $\mathcal{V}$ as in \cite{EMOC}. 
Throughout this section, let $\V$, $\mu_{\V}$, $\V_{L}$ and $\V_{L}^{0}$ be as defined in Section \ref{sec:alg and ellipse}.
The following proposition states the Marcinkiewicz-Zygmund inequality for algebraic polynomials and their gradients.

\begin{prop} \label{th:MZ 2}
Let $\V$ be as in Section \ref{sec:alg and ellipse} and
let $0<c_3<c_4$.
Then, there exists a constant $C=C(\V,c_3,c_4)\geq1$ such that for all $0<a\leq1\leq b$, for all integers $N\ge C b(\frac{b}{a})^{2d} L^d$, for all partitions $\{ R_{j}\}_{j=1}^{N}$ with constants $a$, $b$, $c_3$ and $c_4$ as in Definition \ref{def:partition},  for all $x_{j} \in R_{j}$, for all $P \in \V_{L}^{0}$ it holds
\begin{align}
\Big|\int_{\V} |P(x)| d\mu_{\V}(x)  -  \sum_{j=1}^N \omega_j |P(x_j)| \Big| \leq \frac{1}{2}\int_{\V} |P(x)| d\mu_{\V}(x),\label{eq:th 14}\\
\Big|\int_{\V} |\nabla P(x)| d\mu_{\V}(x)  -  \sum_{j=1}^N \omega_j |\nabla P(x_j)| \Big| \leq \frac{1}{2}\int_{\V} |\nabla P(x)| d\mu_{\V}(x),\label{eq:th 14 b}
\end{align}
where $\omega_j=\mu_{\V}(R_j)$ for all $j=1,\ldots,N$.
\end{prop}

\begin{proof}[Proof of Proposition \ref{th:MZ 2}]
We follow the equal weight case in \cite{EMOC} with minor technical modifications. We start with \eqref{eq:th 14}. 
%
%
Triangle and reverse triangle inequalities lead to
\begin{align*}
\Big| \int_{\V} |P(x)| d\mu_{\V}(x)  -  \sum_{j=1}^N \omega_j |P(x_j)| \Big|& 
= \Big| \sum_{j=1}^N\int_{R_j} |P(x)| d\mu_{\V}(x)  -  \sum_{j=1}^N \int_{R_j} |P(x_j)| d\mu_{\V}(x)\Big|\\
& \leq \sum_{j=1}^N\int_{R_j}\big||P(x)|   -   |P(x_j)|\big| d\mu_{\V}(x)\\
& \leq \sum_{j=1}^N\int_{R_j}\big|P(x)   -  P(x_j)\big| d\mu_{\V}(x).
\intertext{There are $x_j'\in 2X_j$, $j=1,\ldots,N$,  such that $\|\nabla P(x_j')\|\geq \|\nabla P(x)\|$, for all $x\in 2X_j$. Since $\diam(X_j)\leq2c_4b^{1/d}N^{-1/d}$,}
\Big| \int_{\V} |P(x)| d\mu_{\V}(x)  -  \sum_{j=1}^N \omega_j |P(x_j)| \Big|& \leq \sum_{j=1}^N \omega_j\diam(2X_j)\|\nabla P(x_j')\|\\
&\leq 2c_4b^{1/d}N^{-1/d}\sum_{j=1}^N \omega_j\|\nabla P(x_j')\|.
\end{align*}

Let $\mathbb{Y}$ denote the complexification of $\V$, which consists of the complex zeros of the ideal defining $\V$. Assume $N\ge 8^dc_4^db L^d$, so that $2X_j\subset B(x'_j,L^{-1})$. From \cite[Section 2.4]{BC}, see also \cite{EMOC}, we know that since $P\in\V_L^0$,
\begin{equation*}
\|\nabla P(x_j')\|\lesssim  L^{2d+1} \int_{B_{\mathbb{Y}}(x_j',L^{-1})} |P(z)| d\mu_{\mathbb{Y}}(z),
\end{equation*}
where $B_{\mathbb{Y}}(x_j',L^{-1})$ denotes the ball in $\mathbb{Y}$ of radius $L^{-1}$ centered at $x_j'$ and $\mu_{\Y}$ denotes the measure on the complexification. 

If $\pi$ is any permutation of $\{1,\ldots,N\}$, for which 
\begin{equation*}
\bigcap_{i=1}^m B_{\Y}(x_{\pi(i)}',L^{-1})\neq \emptyset,
\end{equation*}
then, as in \cite{EMOC}, for a constant $\beta\geq1$ depending only on $\V$, 
\begin{equation*}
\bigcap_{i=1}^m B(x_{\pi(i)}',\beta L^{-1})\neq \emptyset.
\end{equation*}
Let $z_{\pi}$ be a point in this intersection. A volume comparison implies
\[
m\leq \frac{\mu_{\V}(B(z_\pi,2\beta L^{-1})}{\min_{i=1,\ldots,m}\{\mu_{\V}(Y_{\pi(i)})\}}\leq \frac{2^dc_2\beta^dbN}{c_1c_3^da^2L^d}.
\]
Therefore, we derive
\begin{align*}
\sum_{j=1}^N \int_{B_{\Y}(x_j',L^{-1})} |P(z)| d\mu_{\Y}(z) &\leq \frac{2^dc_2\beta^dbN}{c_1c_3^da^2L^d} \int_{\bigcup_{j=1}^N B_{\Y}(x_{j}',L^{-1})}|P(z)| d\mu_{\Y}(z)
\intertext{According to \cite[Lemma 3.1]{EMOC}, this leads to }
\sum_{j=1}^N \int_{B_{\Y}(x_j',L^{-1})} |P(z)| d\mu_{\Y}(z) &\lesssim \frac{2^dc_2\beta^{d}b}{c_1c_3^da^2}\frac{N}{L^{2d}}\int_{\V} |P(x)|d\mu_{\V}(x),
\end{align*}
so that we obtain
\begin{align*}
\sum_{j=1}^N \omega_j\|\nabla P(x_j')\|&\lesssim\frac{2^dc_2\beta^{d}b^2}{c_1c_3^da^2}L\int_{\V} |P(x)|d\mu_{\V}(x).
\end{align*}
Putting all this together, we derive
\begin{align*}
\Big|\int_{\V} |P(x)| d\mu_{\V}(x)  -  \sum_{j=1}^N \omega_j |P(x_j)| \Big| &\lesssim 
2c_4b^{1/d}N^{-1/d}   \frac{2^dc_2\beta^{d}b^2}{c_1c_3^da^2}L   \int_{\V} |P(x)|d\mu_{\V}(x)\\
&\lesssim \frac{2^{d+1}c_2c_4\beta^db^{2+1/d}}{c_1c_3^da^2}LN^{-1/d}   \int_{\V} |P(x)|d\mu_{\V}(x)\\
&\leq \frac 12 \int_{\V} |P(x)|d\mu_{\V}(x)
\end{align*}
as long as one adjusts the constant in the assumption $N\gtrsim b(\frac{b}{a})^{2d} L^d$. 

The inequality \eqref{eq:th 14 b} follows from \eqref{eq:th 14} as in \cite{EMOC}. We omit the details. 
\end{proof}

\section{Gradient flow}\label{SecP4}

We now apply Propositions \ref{MZM} and \ref{th:MZ 2} to verify Lemma \ref{prop_4}. 
\begin{proof}[Proof of Lemma \ref{prop_4}]
Here we follow closely the proof from \cite[Section 4]{BRV}.
According to Proposition \ref{partition} there exists a partition of $\X$ with constants $a$, $b$, $c_3$ and $c_4$, $\mathcal R=\{R_j\}_{j=1}^N$, such that $\mu_{\M}(R_j)=\omega_j$ for all $j=1,\ldots, N$. 
Recall that $h=1$ if $\X=\M$ and $h=d$ if $\X=\V$.
Let now $N\geq C(\X,c_3,13c_4)b (b/a)^{2h} L^d$, where $C(\X,\cdot,\cdot)$ is as in Propositions \ref{MZM} and \ref{th:MZ 2}.
We start choosing an arbitrary $x_j \in R_j$ for all $1\leq j \leq N$ and consider the map $U:\X_{L}^{0}\to \mathcal{X}(\X)$
\begin{equation*}
U(P)(y)=\frac{\nabla P(y)}{v_{\varepsilon}\left(\|\nabla P(y) \| \right)},
\end{equation*}
where with $\mathcal{X}(\X)$ we call the space of differentiable vector fields on $\X$ and $v_{\varepsilon}$ is as in equation \eqref{eq:v}.
For each $1\leq j \leq N$ let $y_j : \X_{L}^{0} \times [0,+\infty) \rightarrow \X$ be the map satisfying the differential equation
\begin{equation*}
\left \{\begin{array}{ll}
      \dfrac{d}{dt}y_j (P,t) = U(P)(y_j (P,t))\\
      y_j(P,0) = x_j
    \end{array}
  \right.
\end{equation*}
for each $P \in \X_{L}^{0}$. 
For every $P\in \X_{L}^{0}$ and for every $j$, the map $t\to y_{j}(P,t)$ is defined and smooth on the whole real line (see \cite[Theorem 6, p. 147]{Spivak}). Furthemore, $U(P)(y)$ is Lipschitz continuous with respect to $P$. It follows that for each $j$ the map $P\to y_{j}(P,\cdot)$ is continuous in $P$ (see \cite[Corollary 1.6, p. 68]{Lang}).
Now set 
\begin{equation*}
F(P)
=
\left(
x_{1} (P),...,x_{N} (P)
\right)
=
\left(
y_{1} \left( P, 12c_4b^{\frac1d}N^{-\frac{1}{d}} \right) ,...,y_{N} \left( P, 12c_4b^{\frac1d}N^{-\frac{1}{d}}\right)
\right).
\end{equation*}
By the above considerations, we have that $F$ is continuous on $\X_{L}^{0}$. Let us take $P \in \X_{L}^{0}$ such that
\begin{equation*}
\int_{\X} ||\nabla P(x)|| d\mu_{\X}(x) = 1,
\end{equation*}
which means $P \in \partial\Omega$, where $\Omega$ is defined as (\ref{Omega}).
Then we can split
\begin{align*}
\sum_{j=1}^{N} 
\omega_{j}
P(x_j(P))
&=
\sum_{j=1}^{N} 
\omega_{j}
P \left( y_{j} \left( P, 12c_4b^{\frac1d}N^{-\frac{1}{d}} \right) \right)\\
&=
\sum_{j=1}^{N} 
\omega_{j}
P(x_j)
+
\int_{0}^{12c_4b^{\frac1d}N^{-\frac{1}{d}}} 
\frac{d}{dt} 
\left(
\sum_{j=1}^{N} 
\omega_{j}
P(y_{j} \left( P, t \right))
\right)dt.
\end{align*}
Observe first that
\begin{multline*}
\left|
\sum_{j=1}^{N} 
\omega_{j}
P(x_j)
\right|
=
\left|
\sum_{j=1}^{N} 
\int_{R_j}
(P(x_j) - P(x)) d\mu_{\X} (x)
\right|
\leq
\sum_{j=1}^{N} 
\int_{R_j}
\left| P(x_j) - P(x)\right| d\mu_{\X} (x)
\\
\leq
\sum_{j=1}^{N} 
\omega_{j}
{\rm diam}(R_j) \max\limits_{z\in 2X_j} ||\nabla P(z)||
\leq
\frac{2c_4b^{ \frac{1}{d}}}{N^{ \frac{1}{d}}}
\sum_{j=1}^{N} 
\omega_j
 ||\nabla P(z_j)||,
\end{multline*}
where $z_j$ is the point that realizes the maximum and $X_j$ is the geodesic ball described in Definition \ref{def:partition}.
We consider now the partition $\mathcal{R'} = \{ R_{1}' , ... ,R_{N}' \}$ where $R_{j}' = R_{j}\cup \{z_j\}$.
Notice that this is a partition of $\X$ with constants  $a$, $b$, $c_3$, $2c_4$ according to Definition \ref{def:partition}, and $\mu_{\M}(R'_j)=\omega_j$ for all $j=1,\ldots, N$.
Therefore, by Propositions \ref{MZM} and \ref{th:MZ 2} applied to $P\in \partial \Omega$ and the partition $\mathcal{R'}$, we have
\begin{multline*}
\left|
\sum_{j=1}^{N} 
\omega_{j}
P(x_j)
\right|
\leq
\frac{2c_4b^{ \frac{1}{d}}}{N^{\frac{1}{d}}}
\sum_{j=1}^{N} 
\omega_i
 ||\nabla P(z_j)||
 \\
 \leq
\frac{2c_4b^{ \frac{1}{d}}}{N^{\frac{1}{d}}}
\left|
\sum_{j=1}^{N} 
\omega_i
 ||\nabla P(z_j)||
 -
 \int_{\X}
  ||\nabla P(z)|| d\mu_{\X} (z)
\right|
+
\frac{2c_4b^{ \frac{1}{d}}}{N^{ \frac{1}{d}}}
 \int_{\X}
  ||\nabla P(z)|| d\mu_{\X} (z)
\\
\leq
\frac{3c_4b^{ \frac{1}{d}}}{N^{ \frac{1}{d}}}
 \int_{\X}
  ||\nabla P(z)|| d\mu_{\X} (z)
  =
  \frac{3c_4b^{ \frac{1}{d}}}{N^{ \frac{1}{d}}}.
\end{multline*}
Furthemore, for $t \in [0, 12c_4b^{\frac1d}N^{-\frac{1}{d}}]$, we have
\begin{multline*}
\frac{d}{dt} 
\left(
\sum_{j=1}^{N} 
\omega_{j}
P(y_{j} \left( P, t \right))
\right)
=
\sum_{j=1}^{N} 
\omega_{j}
\frac{||\nabla P(y_{j} \left( P, t \right))||^2}{v_\epsilon (||\nabla P(y_{j} \left( P, t \right))||)}
\\
\geq
\sum_{j:||\nabla P(y_{j} \left( P, t \right))|| \geq \epsilon} 
\omega_{j}
||\nabla P(y_{j} \left( P, t \right))||
\geq 
\sum_{j=1}^{N} 
\omega_{j}
||\nabla P(y_{j} \left( P, t \right))||
- \epsilon.
\end{multline*}

Since $|y_{j} \left( P, t \right) - x_j|\leq t$, the partition $\mathcal{R''} = \{ R_{1}'' , ... ,R_{N}'' \}$ where $R_{j}'' = R_{j} \cup \{ y_{j} \left( P, t \right) \}$ is a partition of $\X$  with constants $a$, $b$, $c_3$ and $13c_4$, according to  Definition \ref{def:partition}, and $\mu_{\X}(R''_j)=\omega_j$ for all $j=1,\ldots, N$.
We can therefore now apply Propositions \ref{MZM} and \ref{th:MZ 2} to $P$ and the new partition $\mathcal{R''}$:
\begin{multline*}
\frac{d}{dt} 
\left(
\sum_{j=1}^{N} 
\omega_{j}
P(y_{j} \left( P, t \right))
\right)
\geq 
\sum_{j=1}^{N} 
\omega_{j}
||\nabla P(y_{j} \left( P, t \right))||
- \epsilon
\\
\geq
\int_{\X}
||\nabla P(y)||d\mu_{\X} (y)
-
\left|
\int_{\X}
||\nabla P(y)||d\mu_{\X} (y)
-
\sum_{j=1}^{N} 
\omega_{j}
||\nabla P(y_{j} \left( P, t \right))||
\right|
- \epsilon
\\
\geq
\frac{1}{2} \int_{\X}
||\nabla P(y)||d\mu_{\X} (y)
- \epsilon
=
\frac{1}{2} - \epsilon,
\end{multline*}
for every $P \in \partial \Omega$ and for every $t \in [0, 12c_4b^{\frac1d}N^{-\frac{1}{d}}]$.
In conclusion we obtain
\begin{multline*}
\sum_{j=1}^{N} 
\omega_{j}
P(x_j(P))
=
\sum_{j=1}^{N} 
\omega_{j}
P(x_j)
+
\int_{0}^{12c_4b^{\frac1d}N^{-\frac{1}{d}}} 
\frac{d}{dt} 
\left(
\sum_{j=1}^{N} 
\omega_{j}
P(y_{j} \left( P, t \right))
\right)
\\
\geq
\frac{12c_4b^{ \frac{1}{d}}}{N^{ \frac{1}{d}}}
\left(
\frac{1}{2} - \epsilon
\right)
-
\frac{3c_4b^{ \frac{1}{d}}}{N^{ \frac{1}{d}}}
=
(3 - 12\epsilon)
\frac{c_4b^{ \frac{1}{d}}}{N^{ \frac{1}{d}}}
>
0.
\end{multline*}
\end{proof}

\appendix
\section{Riemannian manifolds}\label{SecM}

Here we present some known results about Riemannian manifolds, we refer the  reader \cite{DC} for a more detailed study on this topic. 

A Riemannian metric on a differentiable manifold $\mathcal M$ is a correspondence which associates to each point $p\in\mathcal M$ an inner product $\langle\cdot,\cdot\rangle_p$ (that is, a symmetric, bilinear, positive definite form) on the tangent space $T_p\mathcal M$ which varies differentiably in the sense that 
if $(U,x)$ is a local chart around $p$ and if $q=x(x_1,\ldots,x_d)$ then 
\[
g_{ij}(x_1,\ldots,x_d)=\left\langle\left(\frac\partial{\partial x_i}\right)_q,\left(\frac\partial{\partial x_j}\right)_q\right\rangle_q
\]
is a differentiable function on $U$. A differentiable manifold with a given Riemannian metric will be called 
a Riemannian manifold. If $v$ is a tangent vector to $\mathcal M$ at $p$, we set $\|v\|^2_p=\langle v,v\rangle_p$. The length of a differentiable curve $\alpha$ from the interval $[a,b]$ to $\mathcal M$
is defined as
\[
\int_a^b\|\alpha'(t)\|_{\alpha(t)}dt.
\]

We define the distance $|p-q|$ between two points in a Riemannian manifold, $p,\,q\in\mathcal M$, as 
the infimum of the lengths of all the differentiable curves joining $p$ and $q$. This is indeed a distance,
and it turns $\mathcal M$ into a metric space that has the same topology as the manifold's natural topology.

If $\mathcal M$ is a compact Riemannian manifold, then for any two points $p$ and $q$ in $\mathcal M$
there exists at least one differentiable curve $\alpha$ joining $p$ and $q$ that realizes the infimum of the lengths of all the differentiable curves joining $p$ and $q$.
A curve $\alpha$ is called geodesic if the covariant derivative of $\alpha'$ along $\alpha$ equals zero (see \cite{DC} for a precise definition of covariant derivative). 
Here we recall only that if $\alpha$ is a geodesic then $\|\alpha'(t)\|_{\alpha(t)}$ is constant, and one can normalize $\alpha$ in such a way that   $\|\alpha'(t)\|_{\alpha(t)}=1$.

It can be shown that there exist two positive constants $c_1$ and $c_2$ such that for any point $x\in\mathcal M$ and for any radius $r\le\mathrm{diam}(\mathcal M)$, the measure of the ball 
$B(x,r)=\{y\in\mathcal M:|x-y|<r\}$ satisfies  the inequalities
\begin{equation*}
c_1 r^d\leq\mu_{\M}(B(x,r))\leq c_2 r^d.
\end{equation*}
It follows easily that there exists a positive constant $c_5$ such that if $f:[0,+\infty)\to[0,+\infty)$ is a decreasing function and if $x\in\mathcal M$ then 
\begin{equation}
\label{radiale}
\int_{\mathcal M}f(|x-y|)d\mu_{\M}(y)\le c_5\int_{0}^{+\infty}f(t)t^{d-1}dt.
\end{equation}	

Let $\mathcal D(\mathcal M)$ the set of functions on $\M$ differentiable in $p$. For any $f\in\mathcal D(\mathcal M)$ we define the gradient of $f$ as a vector field $\nabla f$ on the Riemannian manifold $\mathcal M$ given by 
\begin{equation}\label{grad}
\langle\nabla f(p),v\rangle_p=vf,\quad p\in\mathcal M,\,v\in T_p\mathcal M.
\end{equation}
In local coordinates, the gradient is given by the formula
\[
\nabla f(p)=\sum_{j=1}^d\sum_{i=1}^dg^{ij}\left(\frac{\partial }{\partial x_i} f\right)\frac{\partial }{\partial x_j},
\]
where $g^{ij}$ are the entries of the inverse matrix of $g_{ij}$. It follows from the definition that 
\[
\|\nabla f\|_p=\sup Xf
\]
where the supremum is taken over all the differentiable vector fields $X$ with norm $\|X\|_p\le 1$.

If $f$ is a differentiable function on $\mathcal{M}$ and $\alpha:[0,|p-q|]\to \mathcal M$ is a normalized geodesic joining $p$ and $q$, then 
\begin{align}\label{geodesic}
\nonumber
 |f(p)-f(q)|&=\left|\int_0^{|p-q|}\frac{d}{dt}(f(\alpha(t)))dt\right|=\left|\int_{0}^{|p-q|}\langle\nabla f(\alpha(t)),\,\alpha'(t)\rangle dt\right|\\
&\le |p-q|\sup_{t}\|\nabla f(\alpha(t))\|. 
\end{align}


One can also define the divergence of a differentiable vector field $X$. In local coordinates one has
\[
\mathrm{div} (X)=\frac{1}{\sqrt{\mathrm{det}(g_{ij})}}\sum_{k=1}^d\frac{\partial}{\partial x_k}\left(\sqrt{\mathrm{det}(g_{ij}})X_k\right),\quad X=\sum_{k=1}^d X_k\frac{\partial}{\partial x_k}.
\]

Finally, the Laplace-Beltrami operator $\Delta:\mathcal D(\mathcal M)\to\mathcal D(\mathcal M)$ is defined as
\[
\Delta f=-\mathrm{div}(\nabla f).
\]

The eigenvalues and eigenfunctions of the Laplacian have been deeply studied, see \cite{Z} for a recent survey on this topic.
Here we will just metion a couple of properties that we use on the proofs.
\begin{prop}\label{isometry}
The Laplacian commutes with isometries.
That is, given any two manifolds $\M$, $\mathcal{N}$ and an isometry $u: \M \longrightarrow \mathcal{N}$ between them, we have that
\begin{equation*}
\Delta_{\M} (f\circ u) \circ u^{-1}= \Delta_{\mathcal{N}} f,
\end{equation*}
for all smooth $f:\mathcal{N}\rightarrow\mathbb{C}$.
\end{prop}

\begin{proof}
For this proof and related issues we refer to \cite{Ch}.
\end{proof}

\begin{prop}[Sturm-Liouville's decomposition]\label{SL}
Given a compact Riemannian manifold $\M$ with a normalized measure induced by the Riemannian metric $\mu_{\M}$ such that $\mu_{\M}(\M) = 1$,  there is an orthonormal basis $\{ \varphi_{j} \}_{j=1}^{\infty}$ of eigenfunctions of the Laplacian $\Delta$ with respective eigenvalues $0 \leq \lambda_{1} \leq \ldots $ such that any function $f \in L^{2} (\M,d\mu_{\M})$ can
be written as a convergent series in $L^{2} (\M,d\mu_{\M})$
\begin{equation*}
f =
\sum_{j=1}^{\infty} 
a_{j}
\varphi_{j}
\end{equation*}
for some coefficients $a_{j} \in \mathbb{R}$.
\end{prop}

Finally, let $\mathcal M$ be a compact oriented Riemannian manifold, let $f$ be a differentiable function on $\mathcal{M}$ and $X$ be a differentiable vector field.  The following Green identity (see \cite[page 267]{L}) holds
\begin{equation}
\label{Green}
\int_{\mathcal M} \langle \nabla f, X\rangle d\mu_{\M}=-\int_{\mathcal M}f \mathrm{div}(X)d\mu_{\M}.
\end{equation}


\begin{thebibliography}{99} 


\bibitem{BC}
{R.~J.~Berman \and J.~Ortega-Cerd\`a}, {\em Sampling of real multivariate polynomials and pluripotential theory}. Amer.~J.~Math.  \textbf{140} no.~3 (2018), 789--820.


\bibitem{BRV}
{A. Bondarenko, D. Radchenko, \and M. Viazovska}, {\em Optimal asymptotic bounds for spherical designs}, Ann. Math. \textbf{178} (2013), 443--452.

\bibitem{BCCGST} 
{L. Brandolini , C. Choirat, L. Colzani, G. Gigante, R. Seri, \and G.Travaglini},
{\em Quadrature rules and distribution of points on manifolds}, Ann. Sc. Norm. Sup. Pisa Cl. Sci. (5)
\textbf{XIII} (2014), 889--923.

\bibitem{BGG}
{L. Brandolini, B. Gariboldi \and G. Gigante}, {\em On a sharp lemma of Cassels and Montgomery on manifolds}, preprint.

\bibitem{BEG}
{A. Breger , M. Ehler \and M. Gr\"af}, {\em Quasi Monte Carlo Integration and Kernel-Based Function Approximation on Grassmannians}. In: Pesenson I., Le Gia Q., Mayeli A., Mhaskar H., Zhou DX. (eds) Frames and Other Bases in Abstract and Function Spaces. Applied and Numerical Harmonic Analysis. Birkh\"auser, Cham (2017).

\bibitem{Ch} 
{I. Chavel , B. Randol \and J. Dodziuk}
{\em Eigenvalues in Riemannian Geometry}, Elsevier Science, 1984.

\bibitem{chi} Y. Chikuse, {\em Statistics on special manifolds}, Lecture Notes in Statistics, Springer, New York (2003).

\bibitem{COOLS2003445}
{R. Cools},
{\em An encyclopaedia of cubature formulas}, Journal of Complexity \textbf{19} (3) (2003), 445--453.

\bibitem{DGS}
{P. Delsarte, J. M. Goethals and J. J. Seidel},
{\em Spherical codes and designs}, Geom. Dedicata \textbf{6} (3) (1977), 363--388.

\bibitem{DC} 
{M. P. do Carmo}, \textit{Riemannian Geometry}, Translated from the second Portuguese edition by Francis Flaherty. 
{Mathematics: Theory \& Applications,} Birkh\"auser Boston, Inc., Boston, MA, 1992.


\bibitem{EMOC}
{U. Etayo, J. Marzo, \and J. Ortega-Cerd\`{a}}, {\em Asympotically optimal designs on compact algebraic manifolds}, J. Monatsh. Math. \textbf{186} (2018), 235--248.

\bibitem{FM2010}
{F. Filbir, \and H. N. Mhaskar}, {\em A Quadrature Formula
for Diffusion Polynomials Corresponding to a Generalized Heat Kernel}, J. Fourier Anal. Appl. \textbf{16} (2010), 629--657.

\bibitem{FM2011}
{F. Filbir, \and H. N. Mhaskar}, {\em Marcinkiewicz-Zygmund measures on manifolds},
J. Complexity \textbf{27} (2011), 568--596.

\bibitem{GG}
{B. Gariboldi, \and G. Gigante}, {\em Optimal Asymptotic bounds for designs on manifolds}, preprint.

\bibitem{GL}
{G. Gigante, \and P. Leopardi}, {\em Diameter bounded equal measure partitions of Alfhors regular metric measure spaces}, Discrete Comput. Geom. \textbf{57} (2017), 419--430.

\bibitem{HP}
{P. de la Harpe, \and C. Pache}, {\em Cubature formulas, geometrical designs, reproducing kernels, and Markow operators}, Infinite groups: geometric, combinatorial and dynamic aspects (Basel), vol. 248, Birkh\"auser, (2005), 219--267.  


\bibitem{KM}
{J. Korevaar and J. L. H. Meyers}, {\em Chebyshev-type quadrature on multidimensional domains}, J. Approx. Theory  \textbf{79} (1), (1994), 144--164.

\bibitem{Lang} 
{S. Lang}, Introduction to differentiable manifolds. Second edition. Universitext. Springer-Verlag, New York (2002).

\bibitem{L} 
{J. M. Lee}, Introduction to smooth manifolds,
Second edition, {Graduate Texts in Mathematics,} \textbf{218}, Springer, New York, 2013.

\bibitem{MM}
{M. Maggioni, \and H. N. Mhaskar}, {\em Diffusion polynomial frames on metric measure spaces}, Appl. Comput. Harmon. Anal. \textbf{24} (2008), 329--353.

\bibitem{TEMLYAKOV2003352}
{V. N. Temlyakov}, {\em Cubature formulas, discrepancy, and nonlinear approximation}, Journal of Complexity \textbf{19} (3) (2003), 352 - 391.


\bibitem{brouwer} 
{D. O'Regan, Y. J. Cho, \and Y. Q. Chen}, Topological Degree Theory and Applications,
{Ser. Math. Anal. Appl.}, \textbf{10}, Chapman \& Hall / CRC, Boca Raton, FL, 2006.


\bibitem{sobolev2013theory}
{S. L. Sobolev \and V. L. Vaskevich}, The Theory of Cubature Formulas, 
{Mathematics and Its Applications}, Springer Netherlands, 2013.

\bibitem{Spivak}
{M. Spivak}, A comprehensive introduction to differential geometry. Vol. I. Third edition. Publish or Perish, Inc., Houston, TX (1999).

\bibitem{ST}
{A. H. Stroud}
\textit{Approximate Calculation of Multiple Integrals}, Prentice-Hall, Englewood Cliffs, NJ, 1971.


\bibitem{Z}
{S. Zelditch},
\textit{Eigenfunctions of the Laplacian on a Riemannian Manifold}, CBMS Regional Conference Series in Mathematics, Conference Board of the Mathematical Sciences, 2017.


\end{thebibliography}
\end{document}